\documentclass[a4paper,11 pt,twoside]{amsart}

\usepackage[utf8]{inputenc}
\usepackage[english]{babel}
\usepackage{times}
\usepackage{amsmath}
\usepackage{amsfonts}
\usepackage{amssymb}
\usepackage{amsmath}
\usepackage{amsthm}
\usepackage{enumerate}
\usepackage{hyperref}
\usepackage{fancyhdr}
\usepackage{mathrsfs} % for \mathscr
\usepackage{tikz}

\usepackage{stmaryrd}
\renewcommand{\dots}{\ifmmode\mathinner{\ldotp\kern-0.2em\ldotp\kern-0.2em\ldotp}\else.\kern-0.13em.\kern-0.13em.\fi}
\newtheorem{theorem}{Theorem}[section]
\newtheorem{proposition}[theorem]{Proposition}
\newtheorem{lemma}[theorem]{Lemma}
\newtheorem{corollary}[theorem]{Corollary}
\newtheorem{remark}[theorem]{Remark}

\numberwithin{equation}{section}
\numberwithin{figure}{section}

\definecolor{darkgreen}{rgb}{0.1,0.7,0.1}

\definecolor{darkred}{rgb}{0.7,0.1,0.1}
\definecolor{darkblue}{rgb}{0.1,0.1,0.7}

\newcommand{\E}{\mathbb{E}}
\newcommand{\R}{\mathbb{R}}
\renewcommand{\P}{\mathbb{P}}

\newcommand{\bE}{\mathbf{E}}

\newcommand{\bP}{\mathbf{P}}

\newcommand{\bX}{\mathbf{X}}

\newcommand{\bbE}{\mathbb{E}}

\newcommand{\bbN}{\mathbb{N}}

\newcommand{\bbP}{\mathbb{P}}

\newcommand{\bbR}{\mathbb{R}}

\newcommand{\cA}{\mathcal{A}}
\newcommand{\cB}{\mathcal{B}}
\newcommand{\cC}{\mathcal{C}}
\newcommand{\cD}{\mathcal{D}}
\newcommand{\cE}{\mathcal{E}}
\newcommand{\cF}{\mathcal{F}}
\newcommand{\cG}{\mathcal{G}}
\newcommand{\cH}{\mathcal{H}}

\newcommand{\cL}{\mathcal{L}}

\newcommand{\cP}{\mathcal{P}}
\newcommand{\cQ}{\mathcal{Q}}
\newcommand{\cR}{\mathcal{R}}
\newcommand{\cS}{\mathcal{S}}
\newcommand{\cT}{\mathcal{T}}

\newcommand{\ccC}{\mathscr{C}}

\newcommand{\ccK}{\mathscr{K}}

\newcommand{\gd}{\delta}
\newcommand{\gep}{\varepsilon}       % \ge already exists...

\newcommand{\gO}{\Omega}
\newcommand{\gl}{\lambda}

\newcommand{\ind}{\mathbf{1}}
\DeclareMathOperator{\gap}{\mathrm{gap}}
\newcommand{\lint}{\llbracket}
\newcommand{\rint}{\rrbracket}

\DeclareMathSymbol{\leqslant}{\mathalpha}{AMSa}{"36} % nicer `smaller or equal'
\DeclareMathSymbol{\geqslant}{\mathalpha}{AMSa}{"3E} % nicer `larger or equal'
\DeclareMathSymbol{\eset}{\mathalpha}{AMSb}{"3F}     % nicer `emptyset'
\renewcommand{\leq}{\;\leqslant\;}                   % redef. of < or =
\renewcommand{\geq}{\;\geqslant\;}                   % redef. of > or =
\newcommand{\dd}{\,\text{\rm d}}             % a straight d for differentials
\newcommand{\mintwo}[2]{\min_{\substack{#1 \\ #2}}} % min with 2 lines
\newcommand{\var}{{\rm Var}}

\newcommand{\cc}{\complement}

\renewcommand{\tilde}{\widetilde}

\renewcommand{\geq}{\ge}
\renewcommand{\leq}{\le}

\begin{document}

\title{Spectral gap and cutoff phenomenon for the Gibbs sampler of $\nabla \varphi$ interfaces with convex potential}

\author{Pietro Caputo}
\address{Department of Mathematics and Physics, Roma Tre University, Largo San Murialdo 1, 00146 Roma, Italy.}
\email{caputo@mat.uniroma3.it}
\author{Cyril Labb\'e}
\address{Universit\'e Paris-Dauphine, PSL University, Ceremade, CNRS, 75775 Paris Cedex 16, France.}
\email{labbe@ceremade.dauphine.fr}
\author{Hubert Lacoin}
\address{IMPA, Estrada Dona Castorina 110, Rio de Janeiro, Brasil.}
\email{lacoin@impa.br}

\pagestyle{fancy}
\fancyhead[LO]{}
\fancyhead[CO]{\sc{P.~Caputo, C.~Labb\'e and H.~Lacoin}}
\fancyhead[RO]{}
\fancyhead[LE]{}
\fancyhead[CE]{\sc{Cutoff phenomenon for $\nabla \varphi$ interfaces}}
\fancyhead[RE]{}

\date{\small\today}

\begin{abstract}
We consider the Gibbs sampler, or heat bath dynamics associated to log-concave measures on $\bbR^N$ describing $\nabla\varphi$ interfaces with convex potentials. Under minimal assumptions on the potential, we find that the spectral gap of the process is always given by $\gap_N=1-\cos(\pi/N)$, and that for all $\epsilon\in(0,1)$, its $\epsilon$-mixing time satisfies $T_N(\epsilon)\sim \frac{\log N}{2\gap_N}$ as $N\to\infty$, thus establishing the cutoff phenomenon.   The results  reveal a universal behavior in that they do not depend on the choice of the potential. 
%universal behavior of the spectral gap and the mixing time 

\medskip

\end{abstract}

\maketitle

\setcounter{tocdepth}{1}
\tableofcontents

\section{Introduction}

\subsection{Model and result}
We consider the $1+1$ dimensional interface model defined as follows.  The state space of the interface is defined by
$$\gO_{N}:=\left\{ (x_0,\dots,x_N)\in \bbR^{N+1} \ : \ x_0=0, x_N= 0  \right\}.$$ 
We fix a potential $V\in\ccC$, where $\ccC$ denotes the set of all functions $V:\bbR\to\bbR$ satisfying the following assumptions:
\begin{enumerate}[(i)]
\item\label{ass:conv} $V$ is convex,
\item\label{ass:poly} $V$ grows at most polynomially:  there exist $C>0$ and $K\ge 1$ such that for all $x\in \bbR$, 
\begin{equation}\label{ass:1}
|V(x)| \leq C (1+|x|)^{K}\;.
 \end{equation}
\item\label{ass:nonaff} $V$ is non-affine: namely we have $V'_+ > V'_-$ where
\begin{equation}\label{VminusVplus}
V'_+:=\lim_{x\to \infty}  V(x)/ x \text{ \ and \ } V'_-:=\lim_{x\to -\infty}  V(x)/x\;
\end{equation}
% $e^{-V}$ is integrable. \cyril{Slightly stronger than $V$ non-affine but allows absolute continuity arguments. Actually, one needs $x^2 e^{-V(x)}$ to be integrable for the absolute continuity arguments.}
\end{enumerate}
%Note that since $V$ is convex, the last assumptions only says that $V$ is not afine.
%From now on we fix a potential $V\in\ccC$. 
%We then define the state space
%$$\gO_{N}:=\left\{ (x_0,\dots,x_N)\in \bbR^{N+1} \ : \ x_0=0, x_N= 0  \right\}.$$ 
The $\nabla\varphi$ interface with potential $V$ is the random element of $\Omega_N$ with distribution $\pi_N$, whose density with respect to Lebesgue measure is given by 
 \begin{equation}\label{defpin}
  \frac{\dd\pi_{N}}{ \dd x_1\dots \dd x_{N-1}}= \frac{e^{-H}}{Z_N}\;,
 \end{equation}
% (with some harmless abuse of notation, we have identified $\gO_N$ with $\bbR^{N-1}$)
 where $Z_N$ is the normalization constant and $H=H_{N,V}$ is the Hamiltonian 
 $$ H(x):=\sum_{k=1}^N V(x_i-x_{i-1}).$$
The Gibbs sampler for the measure $\pi_N$ that we wish to consider is the heat-bath dynamics defined as follows.
 Let $\cQ_k$ be the operator that equilibrates the $k$-th coordinate of $x_k$ conditionally given the remaining coordinates.
More precisely, letting $x^{(k,u)}$ denote the vector $(x_0,\dots,x_{k-1},u,x_{k+1},\dots,x_N)$, set
\begin{equation}\begin{split}\label{projectors}
                  \cQ_kf(x)&:= \frac{\int f(x^{(k,u)}) e^{-H(x^{(k,u)})}\dd u}{\ \int  e^{-H(x^{(k,u)})}\dd u }\\
                  &= \int f(x^{(k,u)}) \rho_{x_{k-1},x_{k+1}}(u) \dd u\;,
% &=\frac{\int f(x^{(k,u)}) e^{-V(x_{k+1}-u)-V(u-x_{k-1}) }\dd u}{\ \int  e^{-V(x_{k+1}-u)-V(u-x_{k-1}) }\dd u }
                \end{split}
\end{equation}
where
$$\rho_{b,c}(u) := \frac{e^{-V(u-b)-V(c-u)}}{\int e^{-V(s-b)-V(c-s)} \dd s}\;.$$
Define the Markov generator $\cL=\cL_{N,V}$ by
\begin{equation}
 \cL f:=\sum_{k=1}^{N-1}(\cQ_k f-f).
\end{equation}
Let $\bX^{x}=(\bX^{x}(t))_{t\ge 0}$ be the continuous time Markov chain on $\gO_{N}$ with generator $\cL$ and initial condition $x$. Given $x\in \gO_N$ and $\nu$ a probability measure on $\gO_N$, let 
$P^x_t$ and $P^{\nu}_t$ denote the distribution at time $t$ of the Markov chain with initial condition $x$ and $\nu$ respectively. 
One can describe the evolution of the process as follows: each coordinate of $\bX^{x}(t)$ is updated with rate $1$ independently. When an update is performed at time $t$ for coordinate $k$ the value of $X^x_k$ is resampled according to the conditional equilibrium measure, whose density is $\rho_{b,c}$ with $b=X^x_{k-1}(t)$, $c=X^x_{k+1}(t)$.
%proportional to 
%$e^{-V(X^x_{k+1}(t)-\cdot)-V(\cdot-X^x_{k-1}(t))}.$

Since $\cL$ is a finite sum of orthogonal projectors, it is a bounded self-adjoint operator on $L^2=L^2(\Omega_N,\pi_N)$, and therefore, the corresponding process is reversible with respect to $\pi_N$. %Our first result is the determination of t
The spectral gap of the Gibbs sampler is defined by 
\begin{equation}\label{defgap}
\gap_N = \inf_{f\in L^2:\, \pi_N(f)=0} \frac{\pi_N(f(-\cL f))}{\pi_N(f^2)},
\end{equation}
we use the notation $\pi_N(f) = \int f \dd\pi_N$. 
We do not know whether the operator $\cL$ has pure point spectrum in general, and therefore the spectral gap does not a priori coincide with (the opposite of) some eigenvalue of $\cL$. Our first result  computes the value of $\gap_N$ and shows that it is indeed an eigenvalue. 

\begin{theorem}\label{Th:gap}
For any potential $V\in\ccC$, for all $N\geq 2$, the spectral gap of $\cL$ is given by
$$ \gap_N = 1 -  \cos\left(\frac{\pi}{N}\right)\;,$$
and the function
\begin{equation}\label{deffn}
 f_N(x) = \sum_{k=1}^{N-1} \sin\left(\frac{k\pi}{N}\right) x_k\;,
\end{equation}
 is an eigenfunction of $\cL$ with eigenvalue $-\gap_N$.
\end{theorem}
We remark that the spectral gap of the dynamics is independent of the choice of the potential $V$, as long as $V\in\ccC$, and it coincides with the first Dirichlet eigenvalue of the discrete Laplace operator on the segment $\{1,\ldots,N-1\}$.

Our next results concern the mixing time of the Gibbs sampler.  
Without restriction on the set of possible initial conditions, this mixing time is infinite. Consequently, we restrict ourselves to initial conditions with absolute height at most $N$, and consider the distance to equilibrium at time $t$ from a worst case initial condition: 
\begin{align}\label{defdn}
d_N(t) &:= \sup_{x\in \Omega_N\,:\; |x|_\infty \le N} \| P_N^x(t) - \pi_N\|_{TV}, 
\end{align}
where the total variation distance between two probability measures $\mu,\nu$ on $\Omega_N$ is defined as 
$$
\|\mu-\nu\|_{TV}= \sup_{B\in \cB(\gO_N)} \left(\mu(B)-\nu(B)\right),
$$
the supremum ranging over all Borel subsets of $\gO_N$. 
Note that we do \emph{not} condition the dynamics to keep the height of the interface within $[-N,N]$.

For any $\epsilon\in(0,1)$, the $\epsilon$-mixing time is then defined as
$$ T_N(\epsilon) := \inf\{t\ge 0: d_N(t) < \epsilon\}\;.$$

\begin{theorem}\label{th:main1}
For any $V\in\ccC$, for all $\epsilon \in (0,1)$:
\begin{align}\label{eq:cutoff}
T_N(\epsilon) \sim \frac{\log N}{2 \gap_N}\;.
\end{align}
\end{theorem}
We use the symbol ``$\sim$" for asymptotic equivalence as $N\to\infty$, so that in view of Theorem \ref{Th:gap}, \eqref{eq:cutoff} is equivalent to 
\begin{align}\label{eq:cutoff2}
\lim_{N\to\infty}
\frac{T_N(\epsilon)}{N^2\log N}=\frac1{\pi^2}\;.
\end{align}
Theorem \ref{th:main1} shows that the $\epsilon$-mixing time is, to leading order, insensitive to the threshold parameter $\epsilon$, that is, the Gibbs sampler satisfies the cutoff phenomenon.  Note again the universal behavior, that is the fact that nothing depends on $V$, as long as $V\in\ccC$. 
\begin{remark}
If the restriction on the absolute height $|x|_\infty \le N$ is replaced by $|x|_\infty \le a_N$ with $a_N \gg \sqrt N$ then our proof carries over and yields
$$ T_N(\epsilon) \sim \frac{\log(a_N N^{-1/2})}{\gap_N}\;.$$
For an interpretation of this result, observe that if the initial condition is  $x_i = a_N$ for every $i\in \{1,\ldots,N-1\}$, then $t= \frac{\log(a_N N^{-1/2})}{\gap_N}$ is exactly the time it takes for
$$ \E[f_N(\bX^x(t))] = f_N(x) e^{-\gap_N t}\;,$$
to drop from $f_N(x) = \Theta(N a_N)$ to $\Theta(N^{3/2})$, the latter being the order of fluctuations of $f_N$ at equilibrium.
\end{remark}

\begin{remark}\label{rem:tilt}
The Markov chain can be viewed as taking values in the larger space 
\begin{equation}\label{def:tildeom}
\widetilde \gO_{N}:=\left\{ (x_0,\dots,x_N)\in \bbR^{N+1} \ : \ x_0=0 \right\}. 
\end{equation}
In that case, the value at the endpoint $X^x_N(t)$ remains equal to its  initial value $x_N$ for all $t$. Moreover, we could have fixed the endpoint $x_N = hN$ with $h \ne 0$ and thus have considered the mixing property of the  process within the set
$$\gO_{N,h}:=\left\{ (x_0,\dots,x_N)\in \bbR^{N+1} \ : \ x_0=0, x_N= hN  \right\}.$$
The results of Theorems \ref{Th:gap} and \ref{th:main1} still hold in this more general setting. Indeed, using the transformation $(x_k)\mapsto (x_{k}- kh)$ which maps $\gO_{N,h}$ to $\gO_{N}$, and considering the modified potential $V_h(\cdot):=V(h+\cdot)$, again an element of $\ccC$, we are back to the original setting. In particular, it follows that the spectral gap is {\em independent} of $h$. Concerning the mixing time, one can actually prove that the $N\to\infty$ limit in Theorem \ref{th:main1} holds uniformly over $h$ in compact sets.
\end{remark}

\begin{remark}
Let us comment on our assumptions on the potential  $V\in\ccC$. The convexity hypothesis on $V$ is the most important, and it is required at various points in the proof. In the language of interacting particle systems  it makes the system \textit{attractive}, in the sense that it entails the existence of a coupling that preserves monotonicity (see Lemma \ref{lem:mongc} below).
The assumption \eqref{ass:1} about polynomial growth is merely technical. It helps us obtain certain estimates, and in practice it does not appear to be very restrictive.
 Finally the assumption \eqref{VminusVplus} is the easiest to justify:\ 
 if $V$ is an affine function then the measure $\pi_N$ is not defined since $e^{-H}$ would not be integrable in this case.
 Note that the definition \eqref{defpin} remains unchanged if $V(u)$ is replaced by $V'(u)=V(u)+au+b$, since in that case $H_{N,V'}(x)=H_{N,V}(x)+bN$ is only modified by a constant.
\end{remark}

\subsection{Related works} 
The relaxation to equilibrium of $\nabla\varphi$ interfaces has been the object of many remarkable works in recent years, especially in conjunction with hydrodynamic limits, see e.g.\ \cite{Giacomin,Funaki} and the references therein. In particular, the validity of functional inequalities for the equilibrium measure $\pi_N$ has been explored under various assumptions on the potential $V$. The dynamics considered in these works is usually the conservative diffusion process, namely the Langevin dynamics associated to the Hamiltonian $H$ in the state space $\Omega_N$. For instance, when the potential is a bounded perturbation of a uniformly strictly convex function, then upper and lower bounds of order $N^{-2}$ on the spectral gap of the Langevein diffusion have been obtained  in \cite{Cap}. Moreover, the stronger logarithmic Sobolev inequality has been established by Menz and Otto \cite{OttoMenz}. These results were shown to hold uniformly in the tilt parameter $h$ when the interface endpoint is fixed at $x_N=hN$. The uniformity in $h$ is a consequence of the assumption of uniform strict convexity and it cannot hold for the diffusion process if the potential is only assumed to be convex. In  the non-uniformly convex case, spectral gap bounds with the correct dependence on the tilt $h$ were obtained in \cite{BartheWolff,BartheMilman} for certain potentials such as the solid-on-solid (SOS) potential $V(u)=|u|$. 
While several of the techniques employed in these works carry over to the jump process we consider in this paper,  as far as we know none of the previous works allows one to actually compute the spectral gap as we do here. As discussed in Remark \ref{rem:tilt}, our results hold uniformly for $h$ in a compact set. Comparison with the SOS case studied in \cite{BartheMilman}  shows in particular that the spectral gap of the Gibbs sampler is much less sensitive than the spectral gap of the diffusion process regarding the choice of the tilt parameter $h$. Moreover, as already mentioned our results are largely insensitive to the choice of the potential $V$.   On the other hand, we cannot handle perturbations of a convex potential, since we strongly rely on the FKG inequality and other monotonicity properties which in general do not hold if the convexity assumption is dropped. 

Interface models of the form \eqref{defpin} are also commonly studied in the discrete setting, namely when the heights $x_i$ are restricted to take only integer values, in which case they form natural models for the interface separating two distinct phases in low temperature spin systems; see, e.g.\ \cite{BodineauIoffeVelenik}. For the discrete SOS model, estimates that are tight up to multiplicative constants for the spectral gap and the mixing time of the Gibbs sampler were obtained in \cite{MartinelliSinclair}. We believe that our main results Theorem \ref{Th:gap} and Theorem \ref{th:main1} can be extended to include this case as well, with small modifications in the proof.  Certainly more challenging would be the determination of the spectral gap and mixing time of the  
local dynamics for the discrete SOS interfaces considered e.g.\ in \cite{Posta,CapMarTon}, where only $\pm1$ increments of the height are allowed at each update. 

The problem of determining whether a given Markov chain exhibits the cutoff phenomenon or not keeps attracting a lot of attention. While a general theory is still out of reach, more and more instances of the phenomenon are being understood. Most of the known results concern Markov chains with finite state space, see e.g.\ the monograph \cite{LevPerWil}.  Especially closely related to our analysis here are the results concerning the exclusion process \cite{Lac16,labbe2019cutoff}.  As in our recent work \cite{CLL}, one of the motivations in the present paper is to investigate the phenomenon for Markov chains with continuous state space. Our previous paper \cite{CLL} establishes the cutoff phenomenon for a heat bath dynamics over the simplex, when the target distribution is uniform or some log-concave generalization thereof.  While our assumptions on the potential $V$ here are general enough to handle target distributions $\pi_N$ from a very large family of log-concave measures, we note that they  do not include the measures on the simplex considered in \cite{CLL} since the positivity constraint characterizing the simplex would require dropping the polynomial growth condition.

\subsection{Overview} 

Section \ref{Sec:Tools} introduces several tools and presents some estimates to be used in the sequel. In Section \ref{Sec:Lower}, we establish the lower bound on the mixing time of Theorem \ref{th:main1} by identifying an initial condition for which the process remains far from equilibrium until the putative mixing time: this initial condition is built in such a way that rather explicit computations can be performed on the law of the image through $f_N$ (from \eqref{deffn}) of the process. A first upper bound on the mixing time is obtained in Section \ref{Sec:Upper}: it catches the correct order but not the precise constant. This bound allows us to determine the spectral gap of the generator. In Sections \ref{Sec:UpperTight} and \ref{Sec:Wedge} we refine the upper bound of the previous section by estimating, under some appropriate coupling, the merging time of two processes starting from a (random) `maximal' initial condition and any arbitrary initial condition, and by proving that the process starting from this maximal initial condition reaches equilibrium by the putative mixing time.

\section{Main tools}\label{Sec:Tools}

\subsection{The gradient dynamics}
The process $(\boldsymbol\eta(t))_{t\ge 0}$ defined on $\bbR^N$ by
the increments $$\eta_k(t)=(X_{k}-X_{k-1})(t)$$ is also Markovian. We sometimes use the notation $\boldsymbol\eta(t)=\nabla{\bf X}(t)$. To describe its evolution, we introduce some notation.

Given $a \in \bbR$, we define the resampling potential $W_a$ as
$$ W_a(u) := V\left(a+u\right) + V\left(a-u\right) - 2 V\left(a\right)\;,$$
and set 
\begin{equation}\label{deftheta}
\theta_a(u):=\frac{ e^{-W_a(u)} \dd u}{Z(a)} \quad \text{ with }  \quad Z(a) := \int_{\bbR} e^{-W_a(u)} \dd u\;.
\end{equation}
The function $W_a(u)$ is symmetric (with respect to $u$, but not with respect to $a$ in general), convex and non-negative. It is minimized at $0$ where it admits the value $0$. Our assumption $V\in\ccC$ ensures that $Z(a)$ is finite.
Note also that
\begin{equation}\label{transinv}
\rho_{-a,a}(u)=\theta_a (u)\quad \text{ and } \quad \rho_{b,c}(u)=\theta_{\frac{c-b}{2}}\left(u -\frac{c+b}{2} \right).
\end{equation}

The dynamics of the gradients is then described as follows. For each $k\in \lint 1, N-1\rint$ at rate one 
$(\eta_k,\eta_{k+1})$ jumps to $(\bar \eta_k - U, \bar \eta_k+U)$ where $\bar \eta_k=\frac{\eta_{k+1}+\eta_{k}}{2}$ and $U$ is a r.v. with density $\theta_{\bar  \eta_k}$ . The associated Markov generator is given by
$$ \tilde \cL f(\eta) = \sum_{k=1}^{N-1} \int \big(f(\eta^{(k,u}) - f(\eta)\big) \theta_{\bar \eta_k}(u) \dd u\;,$$
where $f:\bbR^N\mapsto\bbR$ and $\eta^{(k,u)}$, $u\in\bbR$, denotes the vector $\eta$ with the  pair $(\eta_{k-1},\eta_k)$ replaced by $(\bar \eta_k - u, \bar \eta_k+u)$.
Note that the invariant measure $\pi_N$, in terms of the gradient variables $\eta$, is nothing but the product probability measure with density proportional to $\otimes_{k=1}^N e^{-V}$, conditioned on  $\sum_{k=1}^N \eta_k = 0$.

\subsection{The action on linear functions}\label{sec:linear}
The generators $\cL$ and $ \tilde \cL$ take a particularly simple form when applied to linear functions. 
If $g_k$ denotes the coordinate map $g_k:x\mapsto x_k$ then
\begin{equation}\label{Eq:laplace} 
\cL g_k(x) = \frac{x_{k-1} + x_{k+1}}{2}-x_k = \frac12 \Delta x_k\;,
\end{equation}
where $\Delta$ denotes the discrete Laplacian. 
Summation by parts and \eqref{Eq:laplace} 
then show that for every $j\in \lint 1,N-1\rint$ the map $f_N^{(j)}:\Omega_N\mapsto \bbR$ given by 
\begin{equation}\label{defj}
f_N^{(j)} (x) := \sum_{k=1}^{N-1} \sin\left( \frac{j \pi k}{N}\right) x_k\;
\end{equation}
is an eigenfunction of $\cL$  with the eigenvalue $-\lambda_N^{(j)}$ where
\begin{equation}\label{defjj}
\lambda_N^{(j)} := 1 - \cos\left( \frac{j\pi}{N}\right)\;.
\end{equation}
Thus, linear functions form an invariant subspace, and the spectrum of $-\cL$ restricted to that subspace consists of the $N$ eigenvalues 
$$0=:\lambda_N^{(0)}<\lambda_N^{(1)}<\lambda_N^{(2)}<\cdots<\lambda_N^{(N-1)}.$$ 
In the case $j=1$, we simply write $f_N$ for $f_N^{(1)}$ and $\gl_N$ for $ \lambda_N^{(1)}$. 
In particular, it follows that $\gap_N\leq \gl_N$. Theorem \ref{Th:gap} will establish that $\gl_N$ is actually equal to the spectral gap of $\cL$.

\subsection{General spectral gap considerations}
Next, we give a rather general characterization of the spectral gap.  
Consider a reversible Markov process $(X_t)_{t\geq 0}$ on a measurable space $\gO$ with generator $\cL$ and stationary distribution $\pi$. Assume that $\cL$ is self-adjoint in $L^2=L^2(\Omega,\pi)$, and define its spectral gap as
\begin{equation}\label{defgapgen}
\gap = \inf_{f\in L^2:\, \pi(f)=0} \frac{\langle f,-\cL f\rangle}{\langle f,f\rangle},
\end{equation}
where we write $\langle\cdot,\cdot\rangle$ for the scalar product in $L^2$. 
Given a probability measure $\nu$ on $\gO$, we let $P^{\nu}_t$ denote the distribution of $X_t$ starting with initial condition $\nu$.
Finally for a probability measure $\nu\ll \pi$ we let $\|\nu\|_{\infty}$ denote the $L^{\infty}$ norm of density $\dd \nu /\dd\pi$.

\begin{proposition}\label{genresult}
The spectral gap satisfies
 \begin{equation}\label{eq:gapchar}
  \gap= -\sup_{ \|\nu\|_{\infty}<\infty } \limsup_{t\to \infty} \frac{1}{t} \log \| P^{\nu}_t- \pi \|_{TV}.
 \end{equation}
If furthermore $\gO$ is a topological space exhausted by compact sets (and equipped with its Borel $\sigma$-algebra) we can restrict the supremum to $\nu$ with compact support.

\end{proposition}

\begin{proof}
Suppose $\nu$ is a  probability measure on $\gO$ with $\|\nu\|_{\infty}<\infty$. Let $\rho$ and $\rho_t$ denote respectively the density of 
$\nu$ and $P^{\nu}_t$ with respect to $\pi$. Then $\rho_t=e^{t\cL}\rho$, and the spectral theorem implies 
%One inequality is an immediate consequence of the fact that for any $\nu$ we have
 \begin{align*}
\| P^{\nu}_t- \pi \|_{TV}&= \frac{1}{2}\| \rho_t- 1 \|_{L^1(\pi)}\\&\le 
\frac{1}{2}\| \rho_t- 1 \|_{L^2(\pi)}\le  \frac1 2 \|\rho-1\|_{L^2(\pi)} e^{-\gap t}.
 \end{align*}
This proves that the spectral gap is at most the right hand side in \eqref{eq:gapchar}. 
The other inequality requires a bit more work. 
 \medskip
 
 Let us first treat the simpler case where $-\gap$ is an eigenvalue of $\cL$. Let $f$ be a normalized eigenfunction such that $\cL f=-\gap f$. Assume without loss of generality that the  
 positive part $f_+$ satisfies $\| f_+ \|^2_{L^2(\pi)}\ge 1/2$ (if not take the negative part).
 Given $M>0$, consider the bounded density 
 $$\rho_{M}= \frac{f_+\wedge M}{\| f_+\wedge M\|_{L^1(\pi)}}.$$
 By monotone convergence 
 and using $\| f_+ \|_{L^1(\pi)}\leq \| f_+ \|_{L^2(\pi)}\leq 1$,
 $$\lim_{M\to \infty} \langle \rho_{M}, f\rangle= \frac{\| f_+ \|^2_{L^2(\pi)}}{\| f_+ \|_{L^1(\pi)}}\ge \frac12.$$
 Let us thus fix $M$ sufficiently large so that 
 $$\langle \rho_{M}, f\rangle= \langle \rho_{M}-1, f\rangle\ge \frac{1}{3}.$$
Recall that $\| \mu_1 - \mu_2 \|_{TV} = \sup\{ \int h \dd (\mu_1-\mu_2): \|h\|_\infty \le 1\}$. If $\nu$ has density $\rho_{M}$,
 then
 \begin{align}
  \| P^{\nu}_t- \pi \|_{TV} &\ge \int \frac{\rho_M}{\|\rho_{M}\|_\infty} \dd (P^{\nu}_t - \pi) = \int \frac{(\rho_{M}-1)}{\|\rho_{M}\|_\infty} \dd  (P^{\nu}_t - \pi)
 \nonumber\\ &= \frac{\langle e^{t\cL } (\rho_{M}-1), (\rho_{M}-1)  \rangle}{\|\rho_{M}\|_\infty} .
\label{eq:sp1} 
\end{align}
If $f$ is an eigenfunction, then 
 $ \rho_{M}-1=  \langle \rho_{M}-1, f\rangle f 
 +g$, 
 where $g$ is orthogonal to $f$ and  $e^{t\cL } g$ is orthogonal to $f$. Therefore,
 $$
\| P^{\nu}_t- \pi \|_{TV} 
 \ge \frac{\langle \rho_{M}-1, f\rangle^2}{\|\rho_{M}\|_\infty}\,\langle e^{t\cL } f, f  \rangle \ge \frac{1}{9{\|\rho_{M}\|_\infty}} e^{-t\gap }.
 $$
 This implies the desired bound in the case where $-\gap$ is an eigenvalue of $\cL$. 
 
If $-\gap$ is not an eigenvalue we argue as follows. Let ${\mathbf E}_{\delta}$ denote the spectral projector of $-\cL$ associated to the interval $[\gap,\gap+\delta]$, and let $H_\delta$ denote the corresponding closed subspace of $L^2(\pi)$. Suppose that $f$ is normalized and $f\in H_\delta$.  
Let  $\rho_{M}$ be defined as above and notice that \eqref{eq:sp1} continues to hold.  Since
$ \rho_{M}= {\mathbf E}_{\delta} \rho_{M} + g$, 
 where $g\in H_\delta^\perp$ and $e^{t\cL } g\in H_\delta^\perp$, one has
\begin{align}
\langle e^{t\cL } (\rho_{M}-1), (\rho_{M}-1)  \rangle  \ge 
 e^{-(\gap+\delta) t} \| {\mathbf E}_{\delta}\rho_M\|^2_{L^2(\pi)}.
  \end{align}
Since $f$ is normalized and $f\in H_\delta$, one has 
$\| {\mathbf E}_{\delta}\rho_M\|^2_{L^2(\pi)}\geq  \langle \rho_M,f \rangle^2$. In conclusion, we have shown that if $\nu$ has density $\rho_M$ then 
\begin{align*}
  \| P^{\nu}_t- \pi \|_{TV} &\ge \frac{1}{9{\|\rho_{M}\|_\infty}} e^{-t(\gap + \delta) }.
\end{align*}
By the arbitrariness of $\delta$ this implies the desired inequality.  
 If $\gO$ is exhausted by compact sets then we can modify the definition of $\rho_{M}$ to make it compactly supported.
\end{proof}

\subsection{Monotone grand coupling}\label{Sec:GC}

We will consider two partial orders on interface configurations: 
\begin{equation}\label{Eq:order}
\begin{split}
 x \le y  \quad &\Leftrightarrow \quad \forall k \in \lint 0, N \rint,\quad x_k\le y_k.\\ 
  x \preccurlyeq y  \quad &\Leftrightarrow \quad \forall k \in \lint 1, N \rint,\quad (x_k-x_{k-1})\le (y_k-y_{k-1}).
  \end{split}
\end{equation}
Note that $\le$ is a natural partial order in both spaces $\Omega_N$ and  $\widetilde \Omega_N$  while $\preccurlyeq$ is only relevant for  the enlarged space $\widetilde  \Omega_N$(recall the definition in \eqref{def:tildeom}). 

\medskip

We present a global coupling of the trajectories $\bX^x$ (and therefore $\boldsymbol\eta$) starting from all possible initial conditions $x$ which preserves both types of monotonicity. 

\begin{lemma}\label{lem:mongc}
There exists a coupling of $\{\bX^x, x\in\widetilde \Omega_N\}$  such that 
\begin{itemize}
\item If $x  \le y$, then $\bX^x(t)\le \bX^y(t)$ for all $t\ge 0$;
\item If $x \preccurlyeq y$, then $\bX^x(t)\preccurlyeq \bX^y(t)$ for all $t\ge 0$.
\end{itemize}
\end{lemma}
\begin{proof}
The coupling is a version of  the usual graphical construction (see e.g.~\cite{Liggettbook}). To each $k\in\lint 1, N-1\rint$ we associate a Poisson clock process $(\cT^{(k)}_i)_{i\ge 1}$ whose increments are i.i.d.~rate one exponentials, and a sequence $(U^{(k)}_i)_{i\ge 1}$ of i.i.d.\ uniform r.v.
 Then, for every $x\in\widetilde\gO_N$, $(\bX^x(t))_{t\ge 0}$ is a c\`ad-l\`ag process that only evolves at the update times $(\cT^{(k)}_i)_{k\in \lint 1, N-1\rint, i\ge 1}$. More precisely at time $t=\cT^{(k)}_i$, if $U^{(k)}_i=u$ then the $k$-th coordinate is updated as follows
 $$ X_k^x(t)=  F^{-1}_{X_{k-1}(t_-),X_{k+1}(t_-)}(u)  \text{ and } X_j^x(t):= X^x_j(t_{-}) \text{ for } j\ne k\;,$$ 
where for $b,c$ we define $F_{b,c}: \bbR \to [0,1]$ as
$$F_{b,c}(t)=\int_{-\infty}^t \rho_{b,c}(u) \dd u\;.$$
By construction,  the law of $\bX^x$ under this coupling is the desired one.
To check that this coupling preserves the partial order ``$\le$" it is sufficient to check that for every $t\in\bbR$
\begin{equation}\label{firstorder}
 b\le b' \text{ and } c\le c' \quad \Rightarrow \quad
F_{b,c}(t)\ge F_{b',c'}(t)\;.
\end{equation}
For the partial order ``$\preccurlyeq$" it suffices to show that if $c-b\le c'-b'$ then
\begin{equation}\label{secondorder} 
 F_{b,c}(t+b)\ge F_{b',c'}(t+b') \quad \text{ and }  \quad F_{b,c}(c-t)\le F_{b',c'}(c'-t).
\end{equation}
We start with \eqref{firstorder}. As $\rho_{b,c}$ and $\rho_{b',c'}$ are positive, continuous and integrate to the same value, there must exist $u_0$ such that 
$\rho_{b,c}(u_0)=\rho_{b',c'}(u_0)$. Let us show that $u_0$ must satisfy 
\begin{equation}\label{comparr}
 \begin{cases}
  \forall u\le u_0, \quad \rho_{b',c'}(u)\le\rho_{b,c}(u),\\
    \forall u\ge u_0, \quad \rho_{b',c'}(u)\ge\rho_{b,c}(u).
 \end{cases}
\end{equation}
We note that the desired inequality \eqref{firstorder} is a simple consequence of  \eqref{comparr}. 
To prove \eqref{comparr}, we set $W_{b,c}=\log \rho_{b,c}$, and show that 
$W_{b',c'}-W_{b,c}$ is nondecreasing. Indeed, everywhere except on a countable set, $W_{b',c'}-W_{b,c}$ is differentiable and we have by convexity
$$(W_{b',c'}-W_{b,c})'(u)= V'(u-b)-V'(u-b')-V'(c-u)+V'(c'-u)\ge 0.$$
This proves \eqref{comparr}. 
Now \eqref{secondorder} only needs to be proved for $b=b'=0$ by translation invariance. With this in mind, the first inequality in \eqref{secondorder} is a consequence of \eqref{firstorder}. Regarding the second inequality in \eqref{secondorder}, we observe that it is equivalent to
$$ \tilde{F}_{c',0}(-t+c') \ge \tilde{F}_{c,0}(-t+c)\;,$$
where $\tilde{F}$ is the distribution function associated to the potential $\tilde{V}(\cdot) = V(-\cdot)$. The later inequality is then exactly of the same form as the first inequality in \eqref{secondorder}: since $\tilde{V}$ satisfies the same assumptions as $V$ we are done. \end{proof}

\subsection{The sticky coupling}\label{Subsec:sticky}
In this section we construct  a coupling of two trajectories $(\bX^{x}(t))_{t\ge 0}$ and $(\bX^{y}(t))_{t\ge 0}$ which is aimed at minimizing the merging time.
This coupling is also monotone, that is if $x\leq y$ then $\bX^{x}(t)\leq\bX^{y}(t)$ at all times. 

In contrast with that of the previous section, this construction cannot naturally be extending to a grand-coupling.
It
can (and will) also be used for two processes $\bX^{(1)}$ and $\bX^{(2)}$ with initial conditions $\bX^{(1)}(0)$ and $\bX^{(2)}(0)$ sampled according to some  prescribed distributions on $\gO_N$.
%Although this coupling is order preserving, we do not require $x^1\le x^2$.
%$\bX^1 (0)\le \bX^2(0)$.

As  in the previous construction, to each $k\in\lint 1, N-1\rint$ we associate a Poisson clock process $(\cT^{(k)}_i)_{i\ge 1}$ whose increments are i.i.d.\ rate one exponentials.
Let us now describe how the updates are performed. If 
$\cT^{(k)}_i=t$ for some $i$ we resample the values of $X^x_k$ and $X^y_k$.
We use the short hand notation
\begin{equation}
 \rho_x:= \rho_{X^x_{k-1}(t_-),X^x_{k+1}(t_-)}\,,\quad  \rho_y:= \rho_{X^y_{k-1}(t_-),X^y_{k+1}(t_-)}\;,
\end{equation}
and set 
\begin{equation}
 p(t,k):=\int_{\bbR}\rho_x(u)\wedge \rho_y(u) \dd u\;.
\end{equation}
Finally we define three probability measures $\nu_1$, $\nu_2$ and $\nu_3$ with densities proportional to $(\rho_x-\rho_y)_+$,  $(\rho_x\wedge \rho_y)$ and $(\rho_y-\rho_x)_+$ (in the case were $\rho_x=\rho_y$ we can set $\nu_1$ and $\nu_3$ to be the Dirac mass at $0$, or any other arbitrary distribution).
The update then goes as follows
\begin{itemize}
 \item With probability $p=p(t,k)$, we set $X^x_k(t)=X^y_k(t)$, and we draw their common value according to $\nu_2$.
 \item With probability $q=1-p$, we draw $X^x_k(t)$ and $X^y_k(t)$ independently with respective distributions $\nu_1$ and $\nu_3$.
\end{itemize}

To see that this coupling preserves ``$\le$" notice that if the configurations are ordered before the update (or more specifically if  $X^x_{k\pm 1}(t_{-} )\le X^y_{k\pm 1}(t_{-})$) then there exists $u_0$ such that 
$\nu_1$ is supported on $(-\infty,u_0]$ and $\nu_3$ on $[u_0,\infty)$, the latter fact being a direct consequence of \eqref{comparr}.

\begin{remark}
More formally, we can define the coupling using, on top of the clock process, $4$ sequences of independent uniform random variables on $[0,1]$ for each coordinate, from which the updates are defined in a deterministic fashion: we couple if and only if the first uniform is smaller than $p$ and we use the three other uniforms to sample independent random variables with distribution $\nu_1$, $\nu_2$ and $\nu_3$ respectively. 
\end{remark}

\subsection{FKG inequalities}

Recall the partial order $\le$ introduced in \eqref{Eq:order}. We say that $f : \Omega_N \to \R$ is increasing if
$$ x \le y \;\; \Rightarrow \;\; f(x) \le f(y)\;.$$
For two probability measures $\mu,\nu$ on $\Omega_N$, we write $\mu\leq \nu$ and say that $\mu$ is stochastically dominated by $\nu$ if for all increasing $f: \Omega_N \to \R$ one has $\mu(f)\le\nu(f)$. 
We also say that a set $A\subset \Omega_N$ is increasing if the map $\ind_A$ is increasing. Finally, for any two configurations $x,y \in \Omega_N$ we introduce the configurations $\min(x,y), \max(x,y)$ defined as 
$$ \min(x,y)_i := \min(x_i,y_i)\;,\quad \max(x,y)_i := \max(x_i,y_i)\;.$$

\begin{proposition}[FKG inequalities]
\label{prop:fkg}
If $f,g$ are increasing then
$$ \pi_N(fg) \ge \pi_N(f) \pi_N(g)\;.$$
Furthermore if $A,B \subset \Omega_N$ are increasing and satisfy
 $$\left\{ x\in A \ \text{ and } \  y\in B\right\}\;\; \Rightarrow \;\;\min(x,y)\in B,$$
then
\begin{equation}\label{piapib}
\pi_N(\cdot \,|\, A) \ge \pi_N(\cdot \,|\, B)\,.
\end{equation}
\end{proposition}
\begin{proof}
By~\cite[Thm 3]{Preston}, the first part of the statement is granted if we have for all $x,y\in \Omega_N$
\begin{equation}\label{Hmaxmix} H(\max(x,y)) + H(\min(x,y)) \le H(x) + H(y)\;.\end{equation}
The convexity of $V$ is sufficient to ensure this inequality, see for instance~\cite[Appendix B1]{Giacomin}.
We turn to the second part of the statement. Set $\mu_A := \pi_N(\cdot \,|\, A)$ and define $\mu_B$ similarly. The densities of these measures are proportional to $e^{-H_A}, e^{-H_B}$ where
$$ H_A(x) := \begin{cases} H(x)&\mbox{ if } x\in A\;,\\
+\infty&\mbox{ if } x\notin A\;.\end{cases}$$
By~\cite[Prop 1]{Preston}, it suffices to check that for all $x,y \in \Omega_N$
$$ H_A(\max(x,y)) + H_B(\min(x,y)) \le H_A(x) + H_B(y)\;.$$
This is granted by \eqref{Hmaxmix} and the assumption on $A,B$.
\end{proof}
A useful example to keep in mind is as follows. Let $\ccK\subset \{1,\dots,N-1\}$ be a set of labels and define, for some $a\in\bbR$, the sets 
\begin{align}\label{ex:abex}
A_i = \left\{x:\,x_i\geq a\right\}\,,\qquad A = \bigcap_{i\in\ccK}A_i\,,\qquad B = \bigcup_{i\in\ccK}A_i\,. 
\end{align}
Then $A,B\subset \Omega_N$ satisfy the requirement in Proposition \ref{prop:fkg} and the inequality \eqref{piapib} is crucially used in the proof of Proposition \ref{prop:muW} below. 

\subsection{Absolute continuity}\label{sec:ac}
It will be useful to compare the conditional probability measure $\pi_N$ to an unconditional measure under which the increments $\eta_i$ are independent and  have the same mean. We need a preliminary lemma. 

\begin{lemma}\label{lem:mean}
Let $V\in\ccC$ and set $I:= ( V'_-, V'_+)$, where $V'_\pm$ are defined in \eqref{VminusVplus}.
The function $\psi: I\mapsto \bbR$ defined by 
\begin{equation}
 \psi(\gl) = \frac{\int u e^{-V(u)+\gl u} \dd u }{\int e^{-V(u)+\gl u} \dd u}
\end{equation}
is %increasing and 
bijective from $I$ to $\bbR$. %In particular, there exists $\gl \in I$ such that $\psi(\gl)=0$.
%%for $\tilde{V}(x) := V(x) - \gl x$ we have 
%$$\int x e^{-V(x) +\gl x} = 0.$$
\end{lemma}
\begin{proof}
The function $\psi$ is increasing since  for any $\gl\in I$:
\begin{align}\label{eq:leg2}
\psi'(\gl)  
 = \frac{\int (u-\psi(\gl))^2 e^{-V(u)+\gl u} \dd u }{\int e^{-V(u)+\gl u} \dd u}>0.
\end{align}
To prove that $\psi$ is surjective we show that 
$\psi(\gl)\uparrow\infty$ when $\gl\uparrow V'_+$ (a similar argument proves that $\psi(\gl)\downarrow-\infty$ when $\gl\downarrow V'_-$). When $V'_+<\infty$ this follows from the fact that $\psi(\gl)$ is the derivative of $\log \int e^{-V(u)+\gl u} \dd u$ which itself tends to infinity when $\gl\to V'_+$  (by convexity we have that
$V(u)\le  V'_+ u +  V(0)$). When $V'_+=\infty$ it is a standard task to check that $\log \int e^{-V(u)+\gl u} \dd u$ grows superlinearly at infinity.
\end{proof}

As a consequence
there exists $\gl \in I$ such that for $\tilde{V}(x) := V(x) - \gl x$ we have $\int x e^{-\tilde{V}(x)}\dd x = 0$. Note that $\tilde{V}\in\ccC$. %satisfies the same assumptions as $V$.
Let $\nu_N$ be the probability measure under which the r.v.~$\eta_k$, $k\in\lint 1,N\rint$ are i.i.d.~with density proportional to  $e^{-\tilde{V}}$. 
Under $\nu_N$, the expectation of the r.v.~$x_N$ vanishes. The next lemma shows that the law of a fixed proportion  (bounded away from $1$)  of the $\eta_k$'s under $\pi_N$ is  absolutely continuous with respect to \ the law of the same r.v.~under $\nu_N$, uniformly in $N$. The point here is that $\nu_N$ remains a product law and is therefore more tractable.

\begin{lemma}\label{compare}
Fix $a \in (0,1)$ and write $N_a := \lfloor a N \rfloor$ for all $N\ge 1$. There exists a constant $C_a > 0$ such that for all $N\ge 1$ and all positive bounded measurable functions $f:\bbR^{N_a} \to \R_+$ we have
$$ \pi_N[f(\eta_1,\ldots,\eta_{N_a})] \le C_a \,\nu_N[f(\eta_1,\ldots,\eta_{N_a})]\;.$$
\end{lemma}
\begin{remark}
 Note that from exchangeability the above statement is also valid for the functional of an arbitrary subset of the increments of cardinality smaller than $ aN$.
\end{remark}

\begin{corollary}\label{cor:bazics}
There exist two constants $c,C>0$ such that for all $u>0$
 \begin{equation}\label{deviates}
  \pi_N\left( \|x\|_{\infty}\ge u \sqrt{N} \right) \le C N e^{-  c \left[ u^2\wedge(\sqrt{N}u)\right]},
 \end{equation}
and 
\begin{equation}\label{largegrad}
 \pi_N\left(  \max_{i\in \lint 1, N \rint} |\eta_i| \ge u \right) \le N e^{- cu}.
\end{equation}

\end{corollary}
\begin{proof}[Proof of Corollary \ref{cor:bazics}]
For \eqref{largegrad} we apply the lemma to $\ind_{|\eta_i|\ge u}$ and use the union bound. For \eqref{deviates} we only need to prove that 
\begin{equation}\label{toprove}
   \pi_N\left( x_i \ge u \sqrt{N} \right) \le (C/2) e^{-  c \left[ u^2\wedge(\sqrt{N}u)\right]}.
\end{equation}
for $i\le N/2$ (the corresponding lower bound and the case $i\ge N/2$ can be dealt with by symmetry) and use union bound. Lemma \ref{compare} applied to $a=N/2$ allows to prove the bound for $\nu_N$ under which $x_i$ is a sum of IID exponentially integrable random variables. Reproducing the classic upper bound computation in the proof of Cramer's Theorem (see e.g. \cite[Chapter 1]{DZ09}) we have
$$\nu_N\left( x_i \ge u \sqrt{i} \right)\le e^{- i\varphi( u i^{-1/2})}$$
where $\varphi(x):= \max_{t\ge 0} \left( tx-\log \frac{\int e^{-\tilde V(u)+tu} \dd u}{\int e^{-\tilde V(u)} \dd u}\right).$
 Our assumptions on $\tilde V$ imply that $\varphi$ has quadratic behavior at zero. Since in addition $\varphi$ is convex, we have necessarily $\varphi(x)\ge c  x^2\wedge x$ for all $x>0$ (for some positive $c>0$) yielding \eqref{toprove}.
\end{proof}

\begin{proof}[Proof of Lemma \ref{compare}]
Let $\sigma^2$ be the variance of $\eta$ under the measure with density proportional to $e^{-\tilde{V}}$. Let $q_k$ be the density of the random variable $\eta_1+ \ldots +\eta_k$ under $\nu_N$. The Local Limit Theorem~\cite[Th. VII.2.7]{Petrov} gives
$$ \lim_{k\to\infty} \sup_{y\in \bbR} \big|\varepsilon(k,y)\big| = 0\;,$$
where we define
$$
\varepsilon(k,y)= \sigma\sqrt{k} \,q_k(y \sigma\sqrt{k}) - g(y),
$$
and $g$ is the density of the standard Gaussian distribution.
Since $g$ is maximized at $0$, for $k$ sufficiently large we may estimate  
$$
\sup_{z\in\bbR}\sqrt{k}\,q_{k}(z)\leq 2g(0)
$$ 
One can check that, for any $f_0$ which is a bounded measurable function of $x_1,\dots,x_{N-1}$, we have
$$ \pi_N[f_0] = \lim_{\delta\downarrow 0} \frac{\nu_N[f_0\, \ind_{x_N \in [-\delta,\delta]}]}{\nu_N(x_N \in [-\delta,\delta])}\;.$$
Taking $f$ as in the statement of the lemma, we thus get for all $N$ sufficiently large,
\begin{align*}
\pi_N[f(\eta_1,\ldots,\eta_{N_a})] &= \frac{\nu_N\big[f(\eta_1,\ldots,\eta_{N_a}) q_{N-N_a}(- x_{N_a})\big]}{q_N(0)}
\\
&\leq \frac{2g(0)}{\sqrt{N-N_a}q_N(0)} \nu_N\big[f(\eta_1,\ldots,\eta_{N_a})\big]\;.
\\
&\le \frac{4}{\sqrt{1-a}}\,\nu_N\big[f(\eta_1,\ldots,\eta_{N_a})\big]\;.
\end{align*}
The result of the lemma follows by adjusting the value of $C_a$ in order to cover also the small values of $N$.
\end{proof}

\subsection{Technical estimates for the resampling probability}\label{sec:tecnos}
The goal of this subsection is to collect some useful estimates on the resampling distribution of our dynamics. All the constants are allowed to depend on the potential $V\in\ccC$ and on nothing else.
Let us mention before starting that, as a consequence of Assumptions
 (\ref{ass:conv}) and (\ref{ass:poly}) on $V$, we have
\begin{equation}\label{eq:ct22}
 |V'(u)| \leq C(1+|u|)^{K-1}\;,
\end{equation}
for all $u$ where $V$ is differentiable, and by continuity, also for the derivatives on the left and on the right when they differ. All issues concerning differentiability appearing in the proofs below can be resolved in this fashion, so we will not mention them.

\medskip

\noindent Our first estimate guaranties that our distribution is sufficiently spread-out. Recall \eqref{deftheta}.
\begin{lemma}\label{lem:lerho}
There exists a constant $C>0$ such that
\begin{equation} \label{lowZ}
 Z(a)\ge \frac{1}{C(1 \vee |a|^K)} \;.
\end{equation}
As a consequence, we have
\begin{equation}\label{eq:ct3}
\|\rho_{b,c}\|_{\infty} \le C(1 \vee |c-b|^K)\;.
 \end{equation}
\end{lemma}
\noindent
Our second lemma ensures that the distribution $\rho_{b,c}$ displays an exponential decay outside of the interval $[b,c]$.
 \begin{lemma}\label{lem:letail}
There exists positive constants $\alpha$ and $C$ such that for all $s\ge 0$ and all $b, c\in \bbR$
we have

 \begin{equation}\label{eq:tail1}
\int_{ (b \vee c) + s}
^\infty\rho_{b,c}(u)du\leq Ce^{-\alpha s} \,.
 \end{equation}
Symmetrically we have
 \begin{equation}\label{eq:tail2}
\int_{-\infty}
^{ (b \wedge c) - s}\rho_{b,c}(u)\dd u\leq C e^{-\alpha s} \,.
 \end{equation}
 In particular, %if $\var(\rho_{b,c})$ denotes 
 the variance of the random variable with density $\rho_{b,c}$ satisfies, for some possibly different choice of $C$,  for every $b,c\in\bbR$:
\begin{equation}\label{lavar}
 \var(\rho_{b,c}):= \int_{\bbR} \left(u-\frac{b+c}{2}\right)^2 \dd u \le C (|b-c|+1)^2
\end{equation}
\end{lemma}
Finally the third lemma allows us to control the total variation distance between the distributions associated with $\rho_{b,c}$ and $\rho_{b',c'}$.

 \begin{lemma}\label{lem:lequ}
There exists a constant $C$ such that for any $b,b',c,c'$
\begin{equation}\label{eq:q1}
q=\frac{1}{2}\int_\bbR|\rho_{b,c}(u)-\rho_{b',c'}(u)|\dd u\leq C\Delta (1\vee |c-b|^K),
 \end{equation}
 where $\Delta:= (|b'-b|+ |c'-c|)/2$.
\end{lemma}

\begin{proof}[Proof of Lemma \ref{lem:lerho}]
From \eqref{transinv}, it suffices to prove \eqref{eq:ct3} with $a$ and $\theta_a$ instead of $(c-b)$ and $\rho_{b,c}$. Since $W_a(u)\geq W_a(0)=0$, we have $\|\theta_{a}\|_{\infty}=Z(a)^{-1}$ and therefore we only need to prove \eqref{lowZ}. Let $z_a$ be defined as the unique positive solution of
$e^{-W_a(z_a)}= \frac12$. Existence and uniqueness of $z_a$ follow from convexity of $W_a$ and the fact that $W_a$ is minimized at $W_a(0)=0$. 
We have
\begin{equation}\label{eq:ct1}
Z(a)\geq \int_{|u|\leq z_a} e^{-W_a(z_a)}du\geq z_a.
\end{equation}
If $z_a > 1$, then \eqref{lowZ} immediately follows. We now assume that $z_a\leq 1$. Writing 
$$
\log 2= W_a(z_a)=\int_0^{z_a}\left(V'(a+u)-V'(a-u)\right)\dd u \;,
$$
we deduce from \eqref{eq:ct22} that
\begin{equation}\label{eq:ct2}
\log 2 \le 2 z_a\max_{|u-a|\leq 1} |V'(u)| \le 2 C Z(a) (|a|+2)^K \;,
\end{equation}
thus concluding the proof.
\end{proof}

\begin{proof}[Proof of Lemma \ref{lem:letail}]
Using translation invariance \eqref{transinv} we 
only need to prove an upper bound for the tail distribution associated with $\theta_a$, that is, for
$\int_{ |a| + s} \theta_a(u)\dd u.$
Also, at the cost of changing the value of $C$, we can assume that $s\ge s_0$ for some sufficiently large $s_0 \ge 1$ independent of $a$.
Recalling that $\theta_a$ integrates to $1$ and is decreasing on $\bbR_+$, we have
\begin{equation}
 \int_{ |a| + s}^{\infty} \theta_a(u)\dd u\le \frac{ \int^{\infty}_{ |a| + s} \theta_a(u)\dd u}{ \int_{0}^{ |a| + s_0} \theta_a(u)\dd u}\le \frac{1}{|a|+s_0}\int^{\infty}_{ s} \frac{\theta_a(|a|+u)}{\theta_a(|a|+s_0)}\dd u \;.
\end{equation}
We can then conclude if we show that for all $u\ge s_0$
\begin{equation}
 \frac{\theta_a(|a|+u)}{\theta_a(|a|+s_0)}\le C e^{-\alpha u}.
\end{equation}
From our assumptions (i) and (iii) on the potential $V$, we have
$$\lim_{u\to+\infty} V'(u)-V'(-u) \in (0,\infty]\;.$$
Therefore, there exist $\alpha > 0$ and $s_0 > 1$ such that for all $u \ge s_0$, we have $V'(u)-V'(-u) \ge \alpha$. We then compute for all $u\ge s_0$
\begin{align*}
 \partial_u \left[\log \theta_a(|a|+u)\right] &= V'(a-|a|-u)-V'(|a|+a +u) \\ &\le V'(-u)-V'(u)\le -\alpha\;,
\end{align*}
which readily yields
\begin{equation}
\frac{\theta_a(|a|+u)}{\theta_a(|a|+s_0)}\le e^{-\alpha(u-s_0)}\;.
\end{equation}

\end{proof}

\begin{proof}[Proof of Lemma \ref{lem:lequ}]
Note that we may assume $|b'-b|+ |c'-c|\le 1$, otherwise the result is trivial. 
In particular,  $|c'-b'|\le |c-b|+1$.
Using the triangle inequality 
$$|\rho_{b,c}(u)-\rho_{b',c'}(u)|\le | \rho_{b,c}(u)-\rho_{b,c'}(u)|
+| \rho_{b,c'}(u)-\rho_{b',c'}(u)|$$
it is sufficient to treat the case where either $b=b'$ or $c=c'$. By translation invariance we reduce to the case $b=b'=0$ (the case $c=c'$ can be treated symmetrically). 
Interchanging the variables if necessary, we may further assume that  $Z(c/2)\ge Z(c'/2)$. Setting
 \begin{equation*}
 \Gamma_{c,c'}(u)= V(c'-u)-V(c-u)\;.
 \end{equation*}
we observe that 
\begin{align}
\label{eq:q2}
q&= \int_\bbR \rho_{0,c}(u)\left(1-\frac{Z(c)}{Z(c')}e^{-\Gamma_{c,c'}(u)}\right)_+\dd u  \\ &\le 
 \int_\bbR\rho_{0,c}(u)\left(1-e^{-\Gamma_{c,c'}(u)}\right)_+\dd u
  \le \int_\bbR\rho_{0,c}(u)(\Gamma_{c,c'}(u))_+\dd u.
\end{align} 
 Using \eqref{eq:ct22} and $2\Delta=|c-c'|$ we have
 \begin{equation}\label{eq:q21}
 |\Gamma_{c,c'}(u)|\leq C\Delta(|u| + |c| + 1)^{K-1}.
 \end{equation}
 We can  conclude using
 \begin{equation}
   \int_\bbR\rho_{0,c}(u)|u|^{K-1} \dd u \le C (1\vee |c|)^K,
 \end{equation}
 which follows  from Lemma \ref{lem:letail}.
\end{proof}

\section{Lower bound on the mixing time}\label{Sec:Lower}

\begin{proposition}\label{prop:lalb}
 There exists a constant $c>0$ such that, for every $N$ and $t\ge 0$,
 \begin{equation}\label{Eq:lowbo}
   d_N(t) \ge 1- \frac{1}{1+  c N e^{-2\gl_N t}},
 \end{equation}
 where $\gl_N=1-\cos(\pi/N)$.
As a consequence, there exists another constant $C$ such that, for all $\gep\in (0,1)$,
\begin{equation}\label{Eq:lowbo1}
 T_N(\gep)\ge  \frac{1}{2\gl_N} \left(\log N + \log (1-\gep)- C\right).
\end{equation}

\end{proposition}

To prove \eqref{Eq:lowbo} we select a test function $f$ and use the fact that if at time $t$ the value $f({\bf X}(t))$ is far from the equilibrium value $\pi_{N}(f)$ with large probability then $d_N(t)$ must be large. This is implemented by choosing a suitable initial condition and by estimating the first two moments of $f({\bf X}(t))$. This is a variant of Wilson's method  \cite{Wil04}.
As for the exclusion process \cite{Wil04} and for the Beta-sampler on the simplex \cite{CLL}, we take  $f=f_N$, the eigenfunction appearing in Theorem \ref{Th:gap}. 
For the remainder of this section we assume for notational simplicity that $N$ is even and we write $\bX$ for the process started from the random initial condition $\bX(0)$ drawn according to the measure 
\begin{equation}
\varrho_N:=\pi_N\left( \cdot  \ | \ x_{N/2} = N/2, |x|_\infty \le N\right).
\end{equation}

\begin{proposition}\label{prop:meanandvariance}
There exists a constant $C$ such that for every $t\ge 0$
$$ \E[f_N(\bX(t))] \ge C^{-1} N^2 e^{-\lambda_N t}\;,\quad \var[f_N(\bX(t))] \le C N^3.$$
\end{proposition}
\begin{proof}[Proof of Proposition \ref{prop:lalb} using Proposition \ref{prop:meanandvariance}]
By definition, $$d_N(t)\ge \| P^{\varrho_N}_t- \pi_N \|_{TV}.$$ 
From \cite[Proposition 7.12]{LevPerWil} one has
\begin{equation}\label{eq:lpw}
\| P^{\varrho_N}_t- \pi_N \|_{TV} \ge 1- \left(1+ \frac{|\E[f_N(\bX(t))]-\pi_N(f_N)|^2}{2\var(f_N(\bX(t))+2\var_{\pi_N}(f_N) }\right)^{-1},
\end{equation}
where $\var_{\pi_N}(f_N)$ denotes the  the variance of $f_N$ with respect to \ $\pi_N$. Using $\pi_N(f_N)=0$ and Fatou's lemma for weak convergence  to control $\var_{\pi_N}(f_N)$ through the variance $\var(f_N(\bX(t))$ at $t=\infty$, Proposition \ref{prop:meanandvariance} implies the estimate 
\begin{align}
\| P^{\varrho_N}_t- \pi_N \|_{TV}\ge 1- \left(1+  \frac{N e^{-2\gl_N t}}{4C^3}\right)^{-1},
\end{align}
which proves \eqref{Eq:lowbo} with $c=1/4C^3$ if $C$ is the constant in Proposition \ref{prop:meanandvariance}.
The lower bound \eqref{Eq:lowbo1} is a simple consequence of \eqref{Eq:lowbo}.

\end{proof}

\begin{proof}[Proof of Proposition \ref{prop:meanandvariance}]
As $f_N$ is an eigenfunction associated with the eigenvalue $-\gl_N$,  see Section \ref{sec:linear}, the process
$$M_t:=e^{\gl_N t}f_N(\bX(t))$$ is a martingale. In particular,
\begin{equation}\label{eq:lowbo11}
\E[f_N(\bX(t))] =e^{-\lambda_N t}\E[f_N(\bX(0))]=e^{-\lambda_N t}\varrho_N(f_N). 
\end{equation}
Under $\hat \pi_N=\pi_{N}(  \cdot \ |  x_{N/2}=N/2)$, the increments $\eta_i$, $i\in\lint 1,N/2\rint$ are exchangeable and have all mean  $1$. 
The same can be said for   $i\in\lint N/2+1,N\rint$ with mean $-1$. The distribution $\hat \pi_{N}$ restricted to the variables in the first half of the segment is the distribution of $\{\eta_i+1\}$ where the $\eta_i$ are distributed according to the measure  $\pi_{N/2}$ for the shifted potential $ V^+(u)=V(u+1)$, see Remark \ref{rem:tilt}. Similarly, for the second half of the segment with 
$V^-(u)=V(u-1)$. Then, an application of Corollary \ref{cor:bazics} 
shows that for any $a_0 > 0$ there exists $c>0$ such that
 for every $a\ge a_0$ and for all $N$ sufficiently large,
\begin{equation}\label{eq:concentration}
\hat \pi_N(  \max |x_i- \hat\pi_N(x_i)|\ge a N \ ) \le e^{-c a N}\;,
\end{equation}
where $\hat\pi_N(x_i)=i$ if $i\leq N/2$ and $\hat\pi_N(x_i)=N-i$ if $i\geq N/2$.
%Then, an application of Corollary \ref{cor:bazics} 
%shows that, 
% for every $\gep>0$, for $N$ sufficiently large,
%\begin{equation}\label{eq:concentration}
%\hat \pi_N(  \max |x_i- \hat\pi_N(x_i)|\ge \gep N \ ) \le e^{-\delta N},
%\end{equation}
%for some $\delta=\delta(\gep)>0$, where $\hat\pi_N(x_i)=i$ if $i\leq N/2$ and $\hat\pi_N(x_i)=N-i$ if $i\geq N/2$.
Moreover, using the Cauchy-Schwarz inequality and \eqref{deviates} one has $$\hat\pi_N(f_N\,\ind_{\|x\|_\infty > N})= o(N^2).$$
It follows that $\varrho_N(f_N)=\hat\pi_N(f_N\ |\ \|x\|_\infty \le N)$ satisfies %is 
\begin{equation}
 \varrho_{N}( f_N )
 = \frac{2 N^2}{\pi^2}(1+o(1)).
\end{equation}
Combined with \eqref{eq:lowbo11} this proves the desired lower bound on $\E[f_N(\bX(t))]$. 

To control the variance, we write  
\begin{equation}
 \var[f_N(\bX(t))]= e^{-2\gl_N t}  \var[M_t]= e^{-2\gl_N t} \left(\var[M_0]+\bbE \left[\langle M \rangle_t\right]\right),
\end{equation}
where $\langle M \rangle_t$ is the increasing predictable process, or angle bracket, associated to the martingale $M_t$ defined above.
The control of $\var[M_0]=\var_{\varrho_N}(f_N)$ can be obtained by reducing to the measure $\hat\pi_N$ and using Lemma \ref{compare}, considering the cases $i\in \lint 1, N/2\rint$ and $i\in  \lint N/2+1, N\rint$ separately as above. More precisely,
for some constant $C$, for every $i\in\lint 1, N\rint$:
\begin{equation}
\hat \pi_N\left( (x_i- \hat\pi_N(x_i))^2  \right) \le C N.
\end{equation}
Using Cauchy-Schwarz, 
\begin{equation}
\hat \pi_N \left(\Big( f_N(x)-\sum_{i=1}^{N} \hat\pi_N(x_i)\sin(i \pi/N)\Big)^2 \right)\le C N^3.
\end{equation}
Recalling \eqref{eq:concentration}, $\varrho_N$ is obtained by conditioning $\hat\pi_N$ to an event of probability larger than $1/2$, and therefore, using the variational representation for the variance of a random variable $X$, $\var(X)=\inf_{m\in\bbR}\bbE[(X-m)^2]$, one finds
\begin{equation}
 \var_{\varrho_N}(f_N)\le  \varrho_N \left(\Big( f_N(x)-\sum_{i=1}^{N}  \hat\pi_N(x_i)\sin(i \pi/N)\Big)^2\right)\le 2 CN^3.
\end{equation}
The martingale bracket can be given an explicit expression.
The contribution to the bracket of the potential update at site $k$ at time $s$ is bounded by 
\begin{equation}\label{eq:contr}
e^{2\gl_N s}\sin^2(k \pi/N)\,\,\bE\left[\left(X_k(s)-X_k(s^-)\right)^2\right],
\end{equation}
where $\bE[\cdot]$ is the expectation with respect to \ the resampling random variable $X_k(s)$ with distribution $\rho_{X_{k-1}(s^-),X_{k+1}(s^-)}$. 
Notice that
\begin{equation}
X_k(s)-X_k(s^-)=
\frac12(\eta_{k+1}(s^-)-\eta_k(s^-)) -U,
\end{equation}
where $U$ has distribution $\theta_{\bar \eta_k(s^-)}$, see \eqref{transinv}. Using 
 Lemma \ref{lem:letail} to estimate the variance of $U$, we see that \eqref{eq:contr} is bounded above by
\begin{align}\label{eq:contro}
Ce^{2\gl_N s}\left[
1+ \eta_{k}(s^-)^2+\eta_{k+1}(s^-)^2\right],
\end{align}
for some constant $C>0$.
Hence,
\begin{equation}\label{lecrochet}
 \langle M \rangle_t
 \le C \int^t_0 e^{2\gl_N s}\sum_{k=1}^{N-1} \left( 1+\eta_{k}(s)^2+\eta_{k+1}(s)^2\right)\dd s.
\end{equation}
To conclude we prove that there exists $C>0$ such that
\begin{equation}\label{touborne}
\forall N\ge 1,  \forall k\in\lint 1,N\rint, \ \forall s\ge 0, \quad  \E[\eta_k(s)^2] \le C.
\end{equation}
Indeed, \eqref{touborne}  combined with \eqref{lecrochet} yields
\begin{equation}
  e^{-2\gl_N t} \bbE \left[\langle M \rangle_t\right]\le   C N \gl^{-1}_N \le C' N^3.
\end{equation}
By symmetry,  it is sufficient to show \eqref{touborne} for $k\leq N/2$. Moreover, using \eqref{eq:concentration} as above, we may consider the dynamics with initial distribution $\hat\pi_N$ instead of $\varrho_N$.
With slight abuse of notation 
we still use the notation $\bX$ for this process.
We are going to prove a bound for  $\E[\max(\eta_k(s),0)^2]$, the analogous bound for the negative part being proved by a symmetric argument.
Using Lemma \ref{lem:mean}, we fix $\gl$ such that 
\begin{equation}
 (\smallint u e^{-V(u)+\lambda u} \dd u) / ( \smallint  e^{-V(u)+\lambda u} \dd u)=2.
\end{equation}
We consider the measure $\tilde \pi_N$ under which the $\eta_i$ are IID with a distribution whose density with respect to Lebesgue is proportional to $e^{-V(u)+\lambda u}$, and note that  $\tilde \pi_N$ is an invariant measure for the generator $\cL$ in the enlarged state space $\widetilde\Omega_N$.

\medskip

In the enlarged state space, we couple $\bX$ with the process $\bX'$ with initial condition distributed according to $\tilde \pi_N( \cdot | \ x_{N/2}\ge N/2)$. Observe that the law of the increments $(\eta_k)_{k\leq N/2}$ under $\hat\pi_N$ coincides with the law of   $(\eta_k)_{k\leq N/2}$ under $\tilde\pi_N(\cdot \ |\ x_{N/2}= N/2)$. Therefore, by Lemma \ref{lem:mongc}, $\bX$ and $\bX'$ can be coupled in such a way that $\eta_k(s)\le \eta'_k(s)$ for all $s\ge 0$ and $k\leq N/2$.
Hence
\begin{equation}
 \bbE \left[\max(\eta_k(s),0)^2\right]\le  \bbE \left[(\eta'_k(s))^2\right].
\end{equation}
Simple estimates for i.i.d.\ random variables show that $\tilde \pi_N(x_{N/2}\ge  N/2) \ge 1/2$, and therefore, using the invariance of $\tilde \pi_N$:
\begin{equation}
 \bbE \left[\max(\eta_k(s),0)^2\right]\le  2\,\tilde \pi_N(\eta_k^2) =\frac{2\int u^2 e^{-V(u)+\lambda u} \dd u}{\int e^{-V(u)+\lambda u} \dd u}.
\end{equation}

\end{proof}

\section{A first upper bound and the spectral gap}\label{Sec:Upper}

In this section, we establish an upper bound on the total-variation distance to equilibrium that holds for \emph{all} $N\ge 2$. From this bound we will derive the value of the spectral gap of the generator. This upper bound is sharp enough to catch the order of the mixing time when $N\to\infty$ but not the right prefactor: this will be sharpened in the next section.
The main result of this section is formulated as follows. 
For a probability distribution $\nu$ on $\gO_N$ we let $B(\nu)$ denote the following quantity
\begin{equation}\label{eq:bnu}
B(\nu):=\mintwo{x\sim \nu}{x'\sim \pi_N} \sqrt{ \sum_{k=1}^{N-1} \bE\left[  |x_k- x'_k|\right]^2},
\end{equation} 
where $\bE$ denotes the expectation with respect to a coupling of $(x,x')$ with marginals $\nu$ and $\pi_N$, and the minimum is taken over all such couplings.  
\begin{proposition}\label{Prop:RW}
There exists a constant $C>0$ such that for any distribution $\nu$ on $\gO_N$, all $t\ge C \log N$ and all $N\ge 2$
\begin{equation}
\| P_t^{\nu} - \pi_N \|_{TV} \le C  \left( N^{1/2} B(\nu) t^{C} e^{-\gl_N t}+ Ne^{-t}\right)\;,
\end{equation}
where $\gl_N=1-\cos(\pi/N)$. 
\end{proposition}
Before giving the proof of Proposition \ref{Prop:RW} we describe some of its consequences for the spectral gap and the mixing time.

\subsection{Proof of Theorem \ref{Th:gap}}
The upper bound in Proposition \ref{Prop:RW}
is valid for all $N\ge 2$ and for all initial distributions $\nu$, without restrictions on the maximal height. In particular, it
allows us to identify the spectral gap of the generator and prove Theorem \ref{Th:gap}.
We already saw that $f_N$ is an eigenfunction of $-\cL$ associated with $\gl_N$. 
It remains to check that the latter is indeed the spectral gap of $\cL$. Using Proposition \ref{genresult},
it is sufficient to check that for any compactly supported distribution $\nu$ 
$$\limsup_{t\to \infty} \frac{1}{t}\log \| P^{\nu}_t- \pi_N \|_{TV}
\le -\gl_N.$$
This follows from Proposition \ref{Prop:RW} since $B(\nu)<\infty$ if $\nu$ has compact support.

\subsection{A first upper bound on the mixing time}
From the considerations in Section \ref{sec:linear} we obtain the following useful contraction bounds. 

\begin{lemma}\label{lem:contraction}
For any $x,y\in\gO_N$, for all $t\ge 0$: 
\begin{equation}\label{Eq:BdHeat}
\left( \sum_{k=1}^{N-1} \bE[X^x_k(t) - X^y_k(t)] \right)^2 \le  N
e^{-2\gl_N t}\sum_{k=1}^{N-1} (x_k - y_k)^2\;,
\end{equation}
where $\bE$ denotes the expectation with respect to an arbitrary coupling of $\bX^x(t)$ and $\bX^y(t)$.   
Moreover, for any distribution $\nu$ on $\gO_N$ and  $t\ge 0$, the quantity defined in \eqref{eq:bnu} satisfies
\begin{equation}\label{Eq:BdHeat21}
B(P_t^\nu)\le B(\nu) e^{-\gl_N t}.
\end{equation}
\end{lemma}
\begin{proof}
From \eqref{Eq:laplace} we have
\begin{equation}\label{Eq:Heat} \partial_t a(t,k) = \frac12 \Delta a(t,k)\;,\end{equation}
where $a(t,k) := \bE[X^x_k(t) - X^y_k(t)]$. An orthonormal basis for $\Delta$ on the segment $\{1,\dots,N-1\}$
with Dirichlet boundary condition at $0$ and $N$ is given by the eigenfunctions $\varphi^{(j)}$,  $j=1,\dots,N-1$:
\begin{equation}\label{Eq:eigenfcts}
\varphi_k^{(j)} = \sqrt{\frac2N}\sin\left(\frac{ jk \pi}{N}\right)\,, \quad (\Delta \varphi^{(j)})_k= -2\gl_N^{(j)} \varphi_{k}^{(j)},
\end{equation}
where $\gl_N^{(j)}$ is given in \eqref{defjj}. 
Expanding $a(t,\cdot)$ along this basis one obtains 
$$ \sum_{k=0}^N a(t,k)^2 \le e^{-2\gl_N t} \sum_{k=0}^N a(0,k)^2\;,$$
and the bound \eqref{Eq:BdHeat} follows from the Cauchy-Schwarz inequality. To prove \eqref{Eq:BdHeat21} we argue as follows. By definition of $B(\nu)$ we may choose a coupling $\bP_0$ of $(\nu,\pi)$ such that 
\begin{equation}\label{initcoupling}
\sum_{k=1}^{N-1}\bE_0 \big[ | X^{\nu}_k(0)- X^{\pi}_k(0)|\big]^2 =B(\nu)^2.
\end{equation}
Under this coupling we let $Y$ and $W$ denote the upper and lower enveloppe of $\{ \bX^{\nu}(0),\bX^{\pi}(0)\}$, setting
$Y_k=X_k^{\nu}(0)\vee X_k^\pi(0)$ and $W_k=X_k^{\nu}(0)\wedge X_k^\pi(0)$. We have by definition
$$ \sum_{k=1}^{N-1} \bE_0[Y_k - W_k]^2 = B(\nu)^2\;.$$
Now we couple four Markov chains $[\bX^{\nu}(t), \bX^\pi(t), \bX^{Y}(t),\bX^W(t)]_{t\ge 0}$ using the coupling $\bP_0$ to set the initial condition ($\bX^{Y}(0)=Y$ and $\bX^W(0)=W$ respectively) and using the monotone grand coupling from Section \ref{Sec:GC} for the dynamics. We let $\bP$ denote the joint law.
As the initial conditions are ordered we obtain from Lemma \ref{lem:mongc} that  under $\bP$  for any $t\ge 0$ we have
$$\bX^W(t)\le\bX^\nu(t)\le \bX^Y(t) \quad \text{ and } \quad \bX^W(t)\le\bX^\pi(t)\le \bX^Y(t) .$$ Therefore the argument used to prove \eqref{Eq:BdHeat} implies that
 \begin{align}\label{Eq:BdHeatao}
\sum_{k=1}^{N-1} \bE[|X^\nu_k(t) - X^\pi_k(t)|]^2 &\le
\sum_{k=1}^{N-1} \bE[X^Y_k(t) - X^W_k(t)]^2\nonumber \\ &
\leq e^{-2\gl_N t}\sum_{k=1}^{N-1} \bE[Y_k - W_k]^2\nonumber\\& =e^{-2\gl_N t} B(\nu)^2.
\end{align}
By stationarity of $\pi$, under $\bP$ the distribution of $\bX^\nu(t)$ and $\bX^\pi_k(t)$ are respectively $P_t^\nu$ and $\pi$, and \eqref{Eq:BdHeat21} follows.
\end{proof}

Next, we show  that  Proposition \ref{Prop:RW} provides an upper bound on the mixing time which is of order $N^2\log N$. 
 This bound is off by a factor $4$ with respect to Theorem \ref{th:main1}. In the next section we will refine the proof in order to catch the right prefactor.

\begin{corollary}\label{th:corol}
For any $\delta>0$, for all $\gep \in (0,1)$ and all $N\ge N_0(\gep,\delta)$ sufficiently  large 
$$ T_N(\gep) \le  \frac{2+\delta}{\gl_N}  \log N\;.$$
\end{corollary}
\begin{remark}\label{rem:uniformity}
 An important observation here  which is used in Section \ref{sec:censor} is that not only the above estimate is  also valid when the boundary condition $x_N=0$ is replaced by $x_N=hN$ (cf. Remark \ref{rem:tilt}), but it is uniform when $h$ takes value in a compact interval (say $[-C,C]$ for some constant $C>0$). Checking this uniformity is a tedious but rather straightforward procedure. We have chosen to omit it in the proof, but the reader can check that it boils down to making sure that all technical estimates in Section \ref{sec:tecnos} are indeed uniform in this sense.  A second observation (which can, this time, immediately be checked from the proof) is that if the bound on $\|x\|_{\infty}$ is chosen to be $N^{\alpha}$, with $\alpha>1/2$ then the corresponding $\gep$-mixing time is smaller than $\frac{1+\alpha+\delta}{\gl_N}  \log N\;.$
Let us also remark that Corollary \ref{th:corol} is sufficient to establish the so-called pre-cutoff phenomenon, namely the fact that
$$
 \limsup_{N\to\infty}\frac{T_N(\gep)}{T_N(1-\gep)}
$$
is uniformly bounded for $\gep\in(0,1/2)$. 
\end{remark}

\begin{proof}
Consider an initial condition $x\in \gO_N$ such that $\|x\|_\infty \le N$. We have $B(\delta_x)\le C N^{3/2}$ so that a direct application of Proposition \ref{Prop:RW} would yield $T_N(\gep) \le \frac{C'}{\lambda_N} \log N$ for some constant $C'$ depending on $C$ and $N$ large enough. However one can sharpen this upper bound as follows.\\
By \eqref{Eq:BdHeat21} we have for $s\ge t$, $B(P^x_{s-t})\le CN^{3/2} e^{-\gl_N(s-t)}$.
Now using Proposition \ref{Prop:RW} for $\nu=P^x_{s-t}$ we obtain for some new constant $C$:
\begin{equation}
 \| P_s^{x} - \pi_N \|_{TV}= \| P_t^{P^x_{s-t}} - \pi_N \|_{TV} \le C  \left( N^{2} t^{C} e^{-\gl_N s}+ Ne^{-t}\right). 
\end{equation}
Then choosing $s=  \frac{2+\delta}{\gl_N}  \log N$ and $t=(\log N)^2$ we can conclude.
\end{proof}

\subsection{Proof of Proposition \ref{Prop:RW}}
The rest of this subsection is devoted to the proof of Proposition \ref{Prop:RW}. 
We are going to perform the proof for $N\ge 3$ (we require $\gl_N<1$ in \eqref{Eq:EstimateTell}). For $N=2$ since the system equilibrates after one update, we have 
 \begin{equation}
  \| P_t^{\nu} - \pi_2 \|_{TV}\le e^{-t}.
 \end{equation}
Moreover, since the total variation distance $\| P_t^{\nu} - \pi_N \|_{TV} $ is monotone as a function of $t$, we may assume without loss of generality that $t$ is an integer.

Fix $t \in \bbN$ and a distribution $\nu$ on $\gO_N$. For notational simplicity we often write $\pi$ instead of $\pi_N$. We are going to construct a (non-Markovian) coupling  $(\bX^{\nu}(s), \bX^\pi(s))_{s\in [0,t]}$, for the two processes starting with respective distributions $\nu$ and $\pi$. We let $\bbP_t$ denote the law of this coupling.
First we couple the initial conditions $\bX^{\nu}(0), \bX^\pi(0)$ in such a way that \eqref{initcoupling} holds. 
The second ingredient for our coupling is a set of independent, rate $1$, Poisson clocks $(\tau_{k})^{N-1}_{k=1}$ (which are independent of the initial conditions) indexed by coordinates from $1$ to $N-1$ (each $\tau_k$ is considered as a subset of $\bbR_+$ ). These clocks determine the update times for the coordinates of our processes. We then define the random time $\cT$ as the largest integer $\ell$ before $t$ such that all the Poisson clocks $\tau_k$ have rung at least once on $(\ell,t)$. More formally, we set (here $\sup\emptyset=0$)
\begin{equation}
 \cT:= \sup \left\{ \ell\in \lint 0, t\rint \ : \forall k, \ 
 \tau_{k}\cap (\ell,t) \ne \emptyset \right\}.
\end{equation}
Note that we have
\begin{equation}
 \bbP_t(t-\cT=\ell)=\begin{cases}
 		(1-e^{-1})^N &\text{ if }\ell=1\;,\\
		(1-e^{-\ell})^N - (1-e^{-\ell+1})^N &\text{ if }\ell\in \lint 2, t-1\rint\;,\\
		1 - (1-e^{-t+1})^N &\text{ if } \ell=t\;.
               \end{cases}
\end{equation}
Observe that there exists $C>0$ such that for all $N\ge 3$ and all $\ell \ge 0$
\begin{equation}\label{Eq:SimpleTell}
\bbP_t(t-\cT\ge \ell) \le C N e^{-\ell}\;.
\end{equation}
Thus, using the fact that $\bbE[f(Z)]=f(0)+\sum_{k=1}^\infty[f(k)-f(k-1)]\bbP(Z\geq k)$ for non negative integer valued random variables $Z$ and any function $f$, provided that the sum in the r.h.s.\ converges, one has   
\begin{equation}\label{Eq:EstimateTell}
\bbE \left[e^{\gl_N(t-\cT)}\right]\leq 1+ 2CN\sum_{\ell=1}^t  \gl_N e^{\ell (\gl_N-1)}\le C',  
\end{equation}
for some constant $C'>0$. 

Now we perform our coupling as follows
\begin{itemize}
 \item For $s\le \cT$, we use the monotone coupling of Subsection \ref{Sec:GC}\ : At each update time we draw a uniform variable  $U$ and the updated values of $X^{\nu}_k$,  $X^{\pi}_k$ are constructed composing $U$ with the inverse of the conditional distribution function. 
 \item For $s> \cT$, we use the sticky coupling of Subsection \ref{Subsec:sticky}\ :  At each update time we couple $X^{\nu}_k$ \text{ and } $X^{\pi}_k$ with maximal probability.
\end{itemize}

To prove Proposition \ref{Prop:RW}, we introduce the r.v. 
\begin{equation}
A_s:= \sum_{k=1}^{N-1} |X^{x}_k(s)-X^{\pi}_k(s)|\;,\quad s\in [0,t]\;.
\end{equation}

\begin{lemma}\label{Lemma:EstimateBruteGap}
There exist $c',C'>0$ such that for all $N\ge 2$, all $t\ge  \log N$ and all $\ell \in \lint 1,t-1\rint$ we have
\begin{equation}
 \bbP_t\left(  \bX^{\nu}(t)\ne \bX^\pi(t) \ | \cT=\ell \right) \le C'\Big( e^{-c' t^2} + t^{2K+1} \E_t[A_\ell \,|\, \cT = \ell] \Big)\;.
\end{equation}
\end{lemma}
\begin{proof}
For every $k\in \lint 1,N-1\rint$, let us denote by $(t_k^{(i)})_{i=1}^{n_k}$ the ordered set of update times occurring at site $k$ on the time-interval $(\cT,t)$. Let $\tilde\cF$ be the sigma-field generated by all the $(t_k^{(i)})_{i=1}^{n_k}$, $k\in \lint 1,N-1\rint$, and by the processes $\bX^{\nu},\bX^\pi$ up to time $\cT$. Denote by $\tilde{\P}_t$ the associated conditional probability. We are going to show that, for some constant $C>0$, on the event $\{\cT = \ell\}$ we have
\begin{equation}\label{eq:timecondit}
 \tilde\P_t(\bX^{\nu}(t)\ne \bX^\pi(t)) \le C\left(\max_{k\in \lint 1,N-1\rint} n_k\right) \left(e^{-c't^2}+t^{2K} A_{\ell}\right)\;.
\end{equation}
and that
\begin{equation}\label{Eq:Maxcondit}
\bbE_t \Big[\max_{k\in \lint 1,N-1\rint} n_k \ |\ \cT=\ell\Big] \le  Ct\;.
\end{equation}

Let us first show how we conclude from \eqref{eq:timecondit} and \eqref{Eq:Maxcondit}. Since $\{\cT=\ell\}$ is $\tilde\cF$-measurable we have 
\begin{equation}
 \bbP_t\left(\bX^{\nu}(t)\ne \bX^\pi(t) \ | \ \cT=\ell \right) \le C \,\bbE_t \left[  \max_{k\in \lint 1,N-1\rint} n_k\left(e^{-c't^2}+  t^{2K}A_{\ell}\right) \ | \ \cT=\ell \right]
\end{equation}
Observe that $A_{\ell}$ and $\max_{k\in \lint 1,N-1\rint} n_k$ are independent under $\bbP_t( \cdot \ | \ \cT=\ell )$. Therefore we get
$$ \bbP_t\left(\bX^{\nu}(t)\ne \bX^\pi(t) \ | \ \cT=\ell \right) \le C'\big(e^{-c'' t^2} + t^{2K+1} \E_t[A_\ell \,|\, \cT = \ell]\big)\;,$$
as required.

Now let us prove \eqref{eq:timecondit}. We introduce the event
\begin{equation}\begin{split}
\cC_\ell&:= \{ \forall s\in [\ell,t]: \|\nabla \bX^{\pi}_{s} \|_{\infty}\le t^2  \}.
\end{split}
\end{equation}
We note that $\cC_\ell = \cap_{i,k} \cC^{i,k}_\ell$ where
\begin{equation}\begin{split}
\cC^{i,k}_\ell&:= \{ \forall s\in [\ell,t_k^{(i)}): \|\nabla \bX^{\pi}_{s} \|_{\infty}\le t^2  \}.
\end{split}
\end{equation}
We say that the update at time $t_k^{(i)}$ is successful if $X^{\nu}_k(t_k^{(i)})=X^{\pi}_k(t_k^{(i)})$.
We let $\tau$ be the time of the first unsuccessful update among the update times $(t_k^{(i)})_{i=1}^{n_k}$, $k\in\lint 1,N-1\rint$. If all the updates are successful, we set $\tau := t$. We have
$$ \{\bX^{\nu}(t) \ne \bX^\pi(t)\}\cap \{\cT=\ell\} \subset \{\tau < t\}\cap \{\cT=\ell\}\;.$$
Indeed, on the event $\{\tau = t\}\cap\{\cT=\ell\}$, there is at least one update per coordinate on $(\ell,t)$ and all the updates are successful so that the two processes merge by time $t$.
Then we write
\begin{align*}
\tilde\P_t(\tau < t) &= \tilde\P_t(\cup_{i,k} \{\tau = t_k^{(i)}\})\\
&\le \tilde\P_t(\cup_{i,k} (\cC^{i,k}_\ell)^\cc) + \tilde\P_t(\cup_{i,k} \{\tau = t_k^{(i)}\}\cap \cC^{i,k}_\ell)\\
&\le \tilde\P_t(\cC^\cc_\ell) + \sum_{i,k}\tilde\P_t(\{\tau = t_k^{(i)}\}\cap \cC^{i,k}_\ell)\;.
\end{align*}
Using Lemma \ref{lem:lequ}, we have
\begin{align*}
&\tilde\P_t(\{\tau = t_k^{(i)}\}\cap \cC_\ell^{i,k})\\
&= \tilde\E_t\big[\tilde\P_t\big(X_k^{\nu}(t_k^{(i)}) \ne X_k^\pi(t_k^{(i)}) \,|\, \cF_{t_k^{(i)}-}\big)\ind_{\{\tau \ge t_k^{(i)}\}\cap \cC^{i,k}_\ell}\big]\\
&\le \tilde\E_t\big[C \max(1, \|\nabla \bX^{\pi}(t_k^{(i)}-) \|_{\infty})^K    \Delta_k(t_k^{(i)}-) \ind_{\{\tau \ge t_k^{(i)}\}\cap \cC^{i,k}_\ell}\big]\;,
\end{align*}
where
$$ 2\Delta_k(s) := |X^{\pi}_{k-1}(s)-X^{\nu}_{k-1}(s)|+|X^{\pi}_{k+1}(s)-X^{\nu}_{k+1}(s)| \;.$$
On the event $\{\tau \ge t_k^{(i)}\}$, all the updates are successful up to time $t_k^{(i)}$ so that
$$ \Delta_k(t_k^{(i)}-) \le \Delta_k(\ell)\;.$$
Consequently, we have
\begin{align*}
\tilde\P_t(\{\tau = t_k^{(i)}\}\cap \cC^{i,k}_\ell) &\le C t^{2K}  \Delta_k(\ell) \;.
\end{align*}
Putting everything together, we find that on the event $\{\cT=\ell\}$ (which is $\tilde\cF$-measurable):
\begin{align*}
\tilde\P_t\big(\bX^{\nu}(t)\ne \bX^\pi(t)\big) &\le  \tilde\P_t(\{\tau < t\})\\
&\le \tilde\P_t(\cC^\cc_\ell) + (\max_{k\in \lint 1,N-1\rint} n_k) C' t^{2K} A_{\ell}\;.
\end{align*}
To bound the first term, we use stationarity and Corollary \ref{cor:bazics} to obtain
\begin{align*}
\tilde\P_t(\cC^\cc_\ell) &\le \left(\max_{k\in \lint 1,N-1\rint} n_k\right)\, \pi_N(\max_{i\in \lint 1,N\rint} |\eta_i| > t^2/2)\\
&\le \left(\max_{k\in \lint 1,N-1\rint} n_k\right) N e^{-c't^2}\;.
\end{align*}
Since $t\ge \log N$, this yields \eqref{eq:timecondit}.\\

Let us now estimate the conditional expectation of $\max_{k\in \lint 1,N-1\rint} n_k$. Let us first describe the conditional law of the $n_k$'s. Let $G$ be the random number of Poisson clocks that have not rung on $(\ell+1,t)$. On the event $\{\cT=\ell\}$ this number is positive. Given $\{\cT=\ell\}$ the $n_k$'s can be obtained as $G$ i.i.d.~Poisson r.v.~of parameter $1$ conditioned to be positive and $N-1-G$ i.i.d.~r.v.~which are the independent sum of a Poisson r.v.~of parameter $1$ and a Poisson r.v.~of parameter $t-\ell-1$ conditioned to be positive.\\
It is simple to check that the law of a Poisson r.v.~of parameter $q$ conditioned to be positive is stochastically increasing with $q$. As a consequence of these observations, we deduce that $\max_{k\in \lint 1,N-1\rint} n_k$, conditionally given $\{\cT=\ell\}$, is stochastically smaller than $\max_{k\in \lint 1,N-1\rint} Z_k$ where $Z_k$ are i.i.d.~r.v.~obtained as the independent sum of a Poisson r.v.~of parameter $1$ and a Poisson r.v.~of parameter $t-1$ conditioned to be positive. Recalling that a Poisson random variable $W$ with parameter $\gl$ satisfies $\bbP(W\geq k)\leq e^{-k(\log(k/\gl)-1)}$, and that $t\geq \log N$, it is not difficult to check that
\begin{equation}\label{eq:poissonmax}
 \E[\max_{k\in \lint 1,N-1\rint} Z_k] \le Ct\;,
 \end{equation}
for some new constant $C>0$. This implies \eqref{Eq:Maxcondit}.
\end{proof}
We now proceed to the proof of our proposition.
\begin{proof}[Proof of Proposition \ref{Prop:RW}]
We start with an upper bound on the expectation of $A_\ell$ given $\{\cT=\ell\}$ for any $\ell \in \lint 0,t-1\rint$. Since up to time $\cT$ we use the monotone grand coupling, arguing as in the proof of Lemma \ref{lem:contraction}, setting $Y_k=(X^\nu_k\vee X^\pi_k)(0)$ and $W_k=(X^\nu_k\wedge X^\pi_k)(0)$
one obtains 
\begin{align*}
\bbE_t \left[ A_\ell \ |\ \cT=\ell \right]^2 &\le \Big(\sum_{k=1}^{N-1} \E_t[X^Y_k(\ell) -X^W_k(\ell)\ |\ \cT=\ell]\Big)^2\\
&\le  N \sum_{k=1}^{N-1} \E_t[X^Y_k(\ell) -X^W_k(\ell)\ |\ \cT=\ell]^2\\
&\le  N e^{-2\gl_N \ell} \sum_{k=1}^{N-1} \E_t[Y_k -W_k]^2.
\end{align*}
%We note that by \eqref{initcoupling}, we have
%$$ \sum_{k=1}^{N-1} \E_t[(X^\nu_k\vee X^\pi_k)(0) - (X^\nu_k\wedge X^\pi_k)(0)]^2 = B(\nu)^2\;.$$
Therefore, by \eqref{initcoupling}
\begin{equation}\label{Eq:Bell}
\bbE_t \left[ A_\ell \ |\ \cT=\ell \right] \le \sqrt{N} B(\nu) e^{-\gl_N \ell}\;.
\end{equation}
By definition of the total-variation distance we have
\begin{align*} \| P_t^{\nu} - \pi_N \|_{TV} &\le
 \P_t(\bX^{\nu}(t)\ne \bX^\pi(t)) 
 \\&= \sum_{\ell=0}^{t-1} \P_t(\bX^{\nu}(t)\ne \bX^\pi(t) \ |\ \cT=\ell) \P_t(\cT=\ell)\;.
 \end{align*}
We treat separately the case $\ell = 0$ (recall that $\cT=0$ on the event where not all Poisson clocks have rung on $(0,t)$). Using \eqref{Eq:SimpleTell} we have
$$ \P_t(\bX^{\nu}(t)\ne \bX^\pi(t) \ |\ \cT=0) \P_t(\cT=0) \le \P_t(\cT=0) \le CNe^{-t}\;.$$
On the other hand, combining, \eqref{Eq:Bell} and Lemma \ref{Lemma:EstimateBruteGap} we find
\begin{multline*}
\sum_{\ell=1}^{t-1} \P_t(\bX^{\nu}(t)\ne \bX^\pi(t) \ |\ \cT=\ell) \P_t(\cT=\ell) \\ \le Ce^{-c't^2} + C'' N^{1/2} B(\nu) t^{2K+1} e^{-\gl_N t} \bbE \left[ e^{\gl_N (t- \cT)} \right]\;.
\end{multline*}
and we can conclude using \eqref{Eq:EstimateTell}.
\end{proof}

\section{Upper bound on the mixing time}\label{Sec:UpperTight}

\subsection{Proof strategy}
The overall strategy is similar to that in \cite{CLL}. First, we show that the `maximal' evolution gets close to equilibrium by time $\log N / (2\gl_N)$. More precisely, let $\nu^\wedge$ denote the equilibrium measure $\pi$ conditioned to having $x_i\geq N$ for all $i=1,\dots,N-1$. Let $\bX^\wedge$ denote the evolution with initial condition $\nu^\wedge$ and call $P_{t}^\wedge$ its law at time $t$. We have the following result. 
\begin{theorem}\label{th:wedge}
For any $\delta>0$,
$$\lim_{N\to \infty} 
\| P_{t_{\delta}}^\wedge - \pi \|_{TV} = 0\;,$$
where $ t_{\delta}:= (1+\delta) \frac{\log N}{2 \lambda_N}.$
\end{theorem}
\noindent Next, using Theorem \ref{th:wedge} as an input, we compare the evolution $\bX^x$ for an arbitrary initial state $x$ with $\|x\|_\infty\leq N$ to $\bX^\wedge$ and show that they come close in total variation by time $t_{\delta}$.
\begin{theorem}\label{th:xwedge}
For any $\delta>0$,
\begin{equation}\label{eq:thth}
 \lim_{N\to \infty} \sup_{x\in \gO_N : \,\|x\|_{\infty}\le N} \| P_{t_{\delta}}^x - P_{t_{\delta}}^\wedge \|_{TV}=0\;.
\end{equation}
\end{theorem}
The upper bound stated in Theorem \ref{th:main1} follows from the two results above and the triangle inequality. Although Theorem \ref{th:main1} can be deduced from Theorem \ref{th:xwedge} alone, the intermediate result provided by Theorem \ref{th:wedge} is a crucial ingredient in our proof of Theorem \ref{th:xwedge}.

\medskip

Let us briefly explain the importance of Theorem \ref{th:wedge} as an intermediary step. Our proof of Theorem \ref{th:xwedge} is based on a coupling argument that uses monotonicity. For this reason it is important to start with initial conditions that are ordered (for the order on $\gO_N$). This is the case here since the random initial configuration $\nu^\wedge$ is by definition always above $x$ if $\|x\|_{\infty}\le N$ (while using directly $\pi_N$ as an initial condition instead of $\nu^\wedge$ would not work).

\medskip

On the other hand, our proof of Theorem \ref{th:xwedge} also requires to apply the equilibrium estimates of Lemma \ref{compare} to $\bX^\wedge$. It is the double requirement of having a trajectory which is already close to equilibrium and above $\bX^{x}(t)$
which makes Theorem \ref{th:wedge} a necessity.

\medskip

Observe that for all $t$ the density
 $\dd P_{t}^\wedge/ \dd \pi_N$ is an increasing function. This allows for the use of various tools in order to control $\| P_{t_{\delta}}^\wedge - \pi \|_{TV}$, such as the FKG inequality as well as the censoring inequality. Our proof of Theorem \ref{th:wedge} (which is postponed to Section \ref{Sec:Wedge}) is entirely based on these tools and cannot be adapted to an arbitrary initial condition.

\medskip

\subsubsection*{Proof strategy for Theorem \ref{th:xwedge}}
The remainder of this section is devoted to the proof of Theorem \ref{th:xwedge}.
From now on, the processes $\bX^{\wedge}(t)$ and 
$\bX^{x}(t)$ are coupled through the sticky coupling of Subsection \ref{Subsec:sticky} (we denote by $\bbP$ the associated distribution).
To prove Theorem \ref{th:xwedge} we want to estimate the time at which the trajectories  $\bX^{\wedge}(t)$ and 
$\bX^{x}(t)$ merge using the auxiliary function 
\begin{equation}\label{defat}
A_t=\sum_{k=1}^{N-1} (X^{\wedge}_k(t)-X^{x}_k(t))\;,
\end{equation}
which corresponds to the area between the two configurations at time $t$. By monotonicity $A_t\geq 0$ and  the merging time of the two trajectories is the hitting time of $0$ by the random process $A_t$.

\medskip

The control of the evolution of $A_t$ proceeds in several steps.
First we use the heat equation for a time $t_{\delta/2}$ to bring the area $A_t$ between the ordered configurations $X^\wedge_t$ and $X^x_t$ below a first  threshold equal to $N^{3/2-\eta}$ where $\eta>0$ is a parameter that will be taken to be small depending on $\delta$. This step relies on Lemma \ref{lem:contraction}.

\medskip

In a second step, we  show that within an additional time $T=O(N^2)$, with large probability, $A_t$ falls below a second threshold   $N^{-\eta}$.
This is a delicate step, which requires the application of diffusive estimates for super-martingales during a finite sequence of intermediate stages each running for a time $O(N^2)$. It relies tremendously on the specificity of the sticky coupling, and also on the fact that one of the trajectories we are trying to couple is already at equilibrium (cf. Theorem \ref{th:wedge}).

\medskip

The final step brings the area from $N^{-\eta}$ to zero, by using
Proposition \ref{Prop:RW}, the proof of which indicates that after the second threshold has been attained
merging occurs with large probability as soon as every coordinate has been updated once, which by the standard coupon collector argument, takes a time of order $\log N$.

\subsection{Proof of Theorem \ref{th:xwedge}}

We introduce the successive stopping times 
$$ \cT_i:= \inf \{ t \ge t_{\delta/2}: \ A_t\le N^{3/2-i\eta}  \}\;, \quad i\ge 0\;,$$
where $\eta>0$ is a parameter that we are going to choose small enough depending on $\delta$.

\textit{Step 1}: We want to show that by time $t_{\delta}/2$, $A_t$ is much smaller than $N^{3/2-4\eta}$ (here the factor $4$ is present only for technical reason, and can be considered irrelevant since $\eta$ is arbitrary).
\begin{lemma}\label{lem:A}
Setting
$\cA=\cA_N:= \{ \cT_4 = t_{\delta/2}\}\;,$ and fixing $\eta\le \delta/20$ we have
$$ \lim_{N\to\infty}\P(\cA_N) = 1\;.$$
\end{lemma}
\begin{proof}
As in Lemma \ref{lem:contraction}
%Using Cauchy-Schwartz and \eqref{Eq:BdHeat} we have
 \begin{align*}
  \bbE[A_t] %\le\sqrt{N} \sqrt{\sum_{k=1}^{N-1}\left( \bbE\left[  X^{\wedge}_k(t)- X_k^{\wedge}(t)\right]\right)^2 } \\ 
  \le 
  \sqrt{N} \sqrt{\sum_{k=1}^{N-1}\left( \bbE\left[  X^{\wedge}_k(0)- x_k\right]\right)^2 }e^{-\gl_N t}\le 4 N^2 e^{-\gl_N t}.
 \end{align*}
In the last inequality we used the fact that $|x_k|\le N$ (by definition) and the fact that $\bbE\left[  X^{\wedge}_k(0)\right]\le 3N$ (cf. the proof of Proposition \ref{prop:Wt} for this estimate).
Using this estimate for $t=t_{\delta/2}$ we obtain 
$$\bbE[A_t]\le 4N^{\left(3-\delta/2\right)/2}.$$
Since by monotonicity of the coupling, $A_t$ is positive, we can conlude using Markov's inequality.
\end{proof}

\noindent\textit{Step 2}: The aim of the second step is to prove the following estimate
\begin{proposition}\label{propstep2}
Introduce $I:= \min\{i\ge 1: 3/2-i\eta \le -\eta\}$. We have
$$ \lim_{N\to\infty} \P(\cT_I \le t_{\delta/2} + N^2/2) = 1\;.$$
\end{proposition}
To highlight better the main ideas of the proof, we postpone the proof of some of the technical lemmas (namely Lemma \ref{lem:le1}, Lemma \ref{lem:bracket} and Lemma \ref{lem:le3}) to the next subsection and focus on the main steps of the reasoning.
By \eqref{Eq:Heat}, we observe that $A_t$ is a super-martingale. 
More precisely, considering the natural filtration  $(\cF_s)_{s\ge 0}$ associated with the process $(\bX^{\wedge},\bX^{x})$ the conditional version of  \eqref{Eq:Heat} summed along the coordinates yield
\begin{equation}
 \bbE[ A_t \ | \ \cF_s]= A_s  -\int^t_s \bbE[ X^{\wedge}_1(u)-X^{x}_1(u)+ X^{\wedge}_{N-1}(u)-X^{x}_{N-1}(u)]\dd u \le A_s,
\end{equation}
where again we have used the fact that our coupling preserves the ordering.

To prove Proposition \ref{propstep2}, we would like to use diffusive estimates in the form of \cite[Proposition 21]{CLL} but this requires 
a modification of  $(A_t)$ in such a way that it becomes a super-martingale with bounded jumps.
We thus define 
\begin{align*}
\cR_i&:= \inf \{ s\ge \cT_{i-1} \ : \ A_{s}\ge N^{3/2-(i-2)\eta}\}\;,\quad i\ge 1\;,\\
\cQ&:= \inf \{ s\ge t_{\delta/2}:\; \|\nabla\bX_{s}^{\wedge}\|_{\infty}\ge (\log N)^{2}\}\;.
\end{align*}
and 
$\cR:= \inf_{i \in \lint 1, I\rint} \cR_i\wedge \cQ$.
We consider the super-martingale 
$$ M_t := \begin{cases} A_t &\mbox{ if } t<\cR\\
A_\cR \wedge N^{\frac32 - (i-2)\eta} &\mbox{ if } t\ge \cR \mbox{ and } \cR =\cR_i < \cR_{i+1}\;,\\
A_\cR &\mbox{ if }t\ge\cR\mbox{ and } \cR= \cQ < \inf_{i \in \lint 1, I\rint} \cR_i\;.
\end{cases}$$

The construction of $M_t$ is designed so that with large probability it coincides with $A_t$. To show this we 
introduce a collection of events:
%restrict ourselves to the event $\cA\cap \cB$ where 
\begin{align*}
 & \cB= \cB_N := \big\{ \forall t\in [t_{\delta/2},N^3]:\quad \|\nabla\bX_{t}^{\wedge}\|_{\infty} < (\log N)^{2} \big\},\\
   & \cC= \cC_N := \big\{ \forall i\in \lint 4, I\rint, \forall s\ge \cT_{i-1}:\quad  A_s\le N^{3/2-(i-3/2)\eta}   \big\},\\
    &\cD=\cD_N := \big\{ \forall t\in [t_{\delta/2},N^3]:\quad  \max( \| \bX_{t}^{\wedge}\|_{\infty},
     \|\bX_{t}^{x}\|_{\infty})\le \sqrt{N}(\log N)^2 \big\},
\end{align*}
Note that on $\cB\cap \cC$, we have $\cR\ge N^3$. We will show that $\cB,\cC,\cD$ are all very likely. This step of the proof requires Theorem \ref{th:wedge} as an input.

\begin{lemma}\label{lem:le1}
We have $\lim_{N\to\infty} \P(\cB_N\cap \cC_N\cap \cD_N) = 1$.
\end{lemma}

\noindent Then using the method developped in \cite{CLL} we control the increments of  $\langle M \rangle$, which denotes the angle bracket of the martingale part of  $M_t$, between each consecutive $\cT_i$.
\begin{lemma}\label{lem:bracket}
The probability of the event
\begin{equation}
 \cE=\cE_N:=\left\{ \forall i\le I,\  \langle M \rangle_{\cT_i}-\langle M \rangle_{\cT_{i-1}}\le 4  N^{3-2(i-2)\eta}\right\}.
\end{equation}
satisfies $\lim_{N\to \infty} \bbP(\cE_N)=1.$
\end{lemma}
Then in order to compare  ${\cT_i}-\cT_{i-1}$ to $\langle M \rangle_{\cT_i}-\langle M \rangle_{\cT_{i-1}}$, we prove the following estimates on the bracket derivative
\begin{lemma}\label{lem:le3}
 When $\cB\cap \cC \cap \cD$ holds,  for all $t\in [t_{\delta/2},N^3\wedge \cT_I]$ we have
  \begin{equation}
 \partial_t \langle M \rangle_t \ge \frac1{8(\log N)^{C}} \min \left( \frac{M_t}{\sqrt{N}}, \frac{M^2_t}{N}\right),
 \end{equation}
 for some constant $C>0$.
 \end{lemma}
Then we can conclude by simply combining the control we have on the bracket increments, and that on the bracket derivative. The following Lemma, combined with the fact that  $\cA\cap\cB\cap\cC\cap \cD\cap\cE$ holds with large probability, implies Proposition \ref{propstep2}
\begin{lemma}\label{lem:4}
On the event $\cA\cap\cB\cap\cC\cap \cD\cap\cE$
we have $$\forall i\in \lint 5, I\rint, \ \cT_i - \cT_{i-1} \le 2^{-i} N^2.$$
In particular we have $\cT_{I}\le t_{\delta/2}+N^2/2$.
\end{lemma}
\begin{proof}
We work on the event $\cA\cap\cB\cap\cC\cap \cD\cap\cE$. Let $j$ be the smallest $i\ge 5$ such that $\cT_i - \cT_{i-1} > 2^{-i} N^2$ and assume that $j \le I$. Then, $\cT_{j-1}+2^{-j}N^2 \le N^3$ so that by Lemma \ref{lem:le3}
$$ \langle M\rangle_{\cT_{j-1} + 2^{-j} N^2}- \langle M\rangle_{\cT_{j-1}} \ge (\log N)^{-C'} 2^{-j} N^2 (N^{1-j\eta} \wedge N^{2(1-j\eta)})\;,$$  
where we use the fact tht $A_t\ge N^{3/2-j\eta}$ if $t\le \cT_j$ and $M_t=A_t$ on $\cB\cap\cC$. Moreover since we work on $\cE$ we have
$$ \langle M\rangle_{\cT_{j-1} + 2^{-j} N^2}- \langle M\rangle_{\cT_{j-1}} \le 4 N^{3-2(j-2)\eta}\;.$$
These two inequalities are incompatible for $N$ large enough and the lemma is proved.
\end{proof}

\noindent \textit{Step 3}: \ The last step consists in bringing the area to $0$ within a short time after $t_{\delta/2}+N^2/2$. Introduce the event
$$ \cG := \big\{ A_{t_{\delta/2}+N^2/2} \le N^{-\eta/2}\big\}\;.$$
The following estimates can be proved as a variant of Lemma \ref{Lemma:EstimateBruteGap}.
\begin{lemma}\label{lem:5}
There exists $C>0$ such that for any $t\ge \log N$ we have
$$\P( \bX^\wedge(t_{\delta/2}+N^2/2+t) \ne \bX^x(t_{\delta/2}+N^2/2+t) \,|\, \cG) \le C(Ne^{-t} + t^{2K+1} N^{-\eta/2})\;.$$
\end{lemma}

\begin{proof}
This is an adaptation of  the argument in Lemma \ref{Lemma:EstimateBruteGap}. Denote by $(t_k^{(i)})_{i=1}^{n_k}$ the ordered set of updates times occurring at site $k$ on the time-interval $(t_{\delta/2}+N^2/2,t_{\delta/2}+N^2/2+t)$. Let $\tilde{\cF}$ be the sigma-field generated by all the $(t_k^{(i)})$ and by $\bX^\wedge(t_{\delta/2}+N^2/2)$, $\bX^x(t_{\delta/2}+N^2/2)$, and let $\tilde{\bbP}$ be the associated conditional probability. Define $\cH := \{\forall k\in\lint 1,N-1\rint: n_k \ge 1\}$. Then, the very same arguments as in the proof of \eqref{eq:timecondit} show that on the $\tilde{\cF}$-measurable event $\cG\cap\cH$ we have
$$ \tilde{\bbP}( \bX^\wedge(t_{\delta/2}+N^2/2+t) \ne \bX^x(t_{\delta/2}+N^2/2+t) ) \le C\left(\max_{k\in \lint 1,N-1\rint} n_k\right) (e^{-ct^2}+t^{2K} B)\;,$$
with
\begin{align*} B &= \sum_{k=1}^{N-1} \big|X^\wedge_k(t_{\delta/2}+N^2/2) - X^x_k(t_{\delta/2}+N^2/2)\big|\\
&= \sum_{k=1}^{N-1} X^\wedge_k(t_{\delta/2}+N^2/2) - X^x_k(t_{\delta/2}+N^2/2)\;.\end{align*}
Furthermore, given $\cH$, the $n_k$'s are i.i.d.~Poisson r.v.~of parameter $t$ conditioned to be positive. Therefore, reasoning as in \eqref{eq:poissonmax}, for all $t\ge \log N$
$$ \E[\max_{k\in \lint 1,N-1\rint} n_k \ |\ \cH] \le Ct\;.$$
Finally, we have
$$ \P(\cH^\cc) \le N e^{-t}\;.$$
Putting everything together we obtain the stated estimate.
\end{proof}
With the help of this final step, we can conclude the proof.
\begin{proof}[Proof of Theorem \ref{th:xwedge}]
By the Martingale Stopping Theorem, since $(A_t)_{t\ge 0}$ is a supermartingale then $(A_{\cT_I+t})_{t\ge 0}$ is also a càd-làg non-negative super-martingale (for the adequate filtration). A maximal inequality (sometimes referred to as Ville's Maximal Inequality see \cite[Exercise 8.4.2]{Dur19} for the discrete time version and also \cite{ville1939}) 
\begin{equation}\label{villeo1}
 \P\left(\sup_{t\ge 0} A_{\cT_I + t} > N^{-\eta/2}\right) \le \E\left[A_{\cT_I}\right]N^{\eta/2}\;.
 \end{equation}
 Therefore,
\begin{equation}\label{ville1}
 \lim_{N\to\infty} \P\left(\sup_{t\ge 0} A_{\cT_I + t} > N^{-\eta/2}\right) = 0\;.
 \end{equation}
Combining this with Lemma \ref{lem:4}, we deduce that the probability of the event $\cG$ goes to $1$.
Applying Lemma \ref{lem:5} we thus deduce that for $t=2\log N$ we have
$$ \lim_{N\to\infty} \P( \bX^\wedge(t_{\delta/2}+N^2/2+t) \ne \bX^x(t_{\delta/2}+N^2/2+t) \,|\, \cG) = 0\;.$$
Since all our estimates hold uniformly over all $x\in \Omega_N$ with $\|x\|_\infty\le N$, this suffices to deduce \eqref{eq:thth}.
\end{proof}

\subsection{Proof of the technical estimates of step 2}

\begin{proof}[Proof of Lemma \ref{lem:le1}]
To prove that $\cB$ and $\cD$ have small probability, we are going to show that similar events have small probability for the stationary version of our Markov chain  $(\bX^{\pi}(t))_{t\ge 0}$  and then use Theorem \ref{th:wedge}.
By a simple coupling argument, for any $\cA\subset \gO_N$ we have
\begin{multline}\label{Eq:wut}
 \P\left(\exists t\in [t_{\delta/2},N^3]:\, \bX_{t}^{\wedge}\in \cA \right)\\ \le \| P_{t_{\delta/2}}^\wedge - \pi\|_{TV} + \P\Big( \exists t\in [0,N^3-t_{\delta/2}]:\,  \bX_{t}^{\pi}\in \cA  \Big)\;,
\end{multline}
where, with slight abuse of notation we denote by $\bbP$ the distribution of $\bX^{\pi}$, the Markov chain starting from the equilibrium distribution.

By symmetry arguments (using the fact that $\hat V(x):=V(-x)$ satisfies $\hat V\in\ccC$), \eqref{Eq:wut} remains true upon replacing $\bX^\wedge$ by $\bX^\vee$ the dynamics with initial distribution $\pi( \cdot \ | \ \forall i\in \lint 1,N-1\rint,\;,  x_i\leq -N)$. 

\medskip

The first term in the r.h.s.\ of \eqref{Eq:wut}  goes to zero by Theorem \ref{th:wedge}. To bound the second term, we use a standard subdivision scheme and estimates on the invariant measure. More precisely, if one subdivides $[0,N^3]$ into intervals of length $N^{-6}$ then with a probability $1-O(N^{-1})$, there are at most one resampling event per interval. Since the process is stationary, we can bound the second term in the r.h.s.\ of \eqref{Eq:wut}  by 
$ N^9 \pi_N(\cA)+C N^{-1}$.
To prove that $\lim_{N\to\infty}\bbP[\cB^{\cc}_N]=0$ use \eqref{Eq:wut} with $\cA= \{ \|\nabla x\|_{\infty} > (\log N)^{2} \}$,
and apply  Corollary \ref{cor:bazics} which entails that $ N^9 \pi_N(\cA)\le N^{-1}$.

\smallskip

We turn now to $\cD$. Using  \eqref{Eq:wut} and the argument above with  $\cA= \{\|x\|_{\infty} > \sqrt{N}(\log N)^2\}$  and  Corollary \ref{cor:bazics} we deduce that
\begin{equation}\label{Eq:Decaywedge} \lim_{N\to\infty}\P\big( \exists t\in [t_{\delta/2},N^3]:\,  \| \bX_{t}^{\wedge}\|_{\infty} > \sqrt{N}(\log N)^2 \big) = 0\;.\end{equation}
and similarly for $\bX^{\vee}$. To get a similar estimate for $\bX^x$
it is sufficient to observe that from Lemma \ref{lem:mongc}
$\bX^x$ is stochastically dominated by $\bX^\vee$
and stochastically dominates $\bX^\vee$, so that we can deduce from \eqref{Eq:Decaywedge}  the desired bound for  $\max_k X_{k}^{x}(t)$ and $\min_k X_{k}^{x}(t)$ respectively, concluding the proof of $\lim_{N\to\infty}\P(\cD_N)=1$.

\smallskip

Finally let us focus on the event $\cC_N$.
For every $i\ge 1$, by the Martingale Stopping Theorem and Ville's Maximal Inequality (as in \eqref{ville1}) we have
$$ \P\left(\sup_{t\ge 0} A_{\cT_{i-1}+t} > N^{3/2-(i-3/2)\eta}\right) \le \E[A_{\cT_{i-1}}] N^{-3/2+(i-3/2)\eta} \le N^{-\eta/2}\;.$$
Since $I$ is a fixed non-random integer, a union bound shows that $\lim\limits_{N\to\infty} \P(\cC_N) = 1$.

\end{proof}

\begin{proof}[Proof of Lemma \ref{lem:bracket}]
The proof follows from a diffusivity bound developped in an earlier work \cite[Proposition 21]{CLL}, applied to the super-martingales 
  $$M^{(i)}_s= M_{s+\cT_{i-1}},$$
 whose jump sizes are bounded above by $N^{3/2-(i-2)\eta}$.
 We refer to \cite{CLL} for more intuition about this inequality.
\end{proof}

To prove Lemma \ref{lem:le3} we will require an intermediate technical result derived from the preliminary work of Section \ref{sec:tecnos} which allows us to estimate the bracket derivative.
Define $$ \delta X_k(t) := X^\wedge_k(t) - X^x_k(t).$$

\begin{lemma}\label{lem:le2}
  When $\cB\cap \cC$ holds, then for all $t\in [t_{\delta/2},N^3 \wedge \cT_I]$ where 
  $\partial_t \langle M \rangle_t$ is differentiable (all $t$ except a random countable set)
 \begin{equation}
 \partial_t \langle M \rangle_t \ge \frac{1}{2}\sum_{k=1}^{N-1}\left[ (\delta X_k(t))^2\wedge (\log N)^{-C_K}\right],
 \end{equation}
for some constant $C_K > 0$. 
\end{lemma}

\begin{proof}[Proof of Lemma \ref{lem:le3} assuming Lemma \ref{lem:le2}]
Write $A_t=U_t + V_t$ where, for some $a>0$: 
\begin{equation}
U_t=\sum_{k=1}^{N-1} 
\delta X_k(t)\ind_{\{\delta X_k(t)< a\}}\,,\quad V_t=\sum_{k=1}^{N-1} 
\delta X_k(t)\ind_{\{\delta X_k(t)\ge a\}}.
 \end{equation}
The Cauchy-Schwarz inequality shows that $$U_t^2\leq N \sum_{k=1}^{N-1} 
(\delta X_k(t))^2\ind_{\{\delta X_k(t)\leq a\}}.$$ 
Take $a=(\log N)^{-\frac12 C_K}$. If $U_t\geq A_t/2$, then 
Lemma \ref{lem:le2} implies
  \begin{equation}
\partial_t \langle M \rangle_t \ge \frac{A_t^2}{8N}.
%\frac{n_t(\log N)^{-C_K}}{2}.
 \end{equation}
If on the other hand $V_t\geq  A_t/2$, then letting  $n_t$ denote the number of indices $k$ such 
that $\delta X_k(t)\ge a$,
Lemma \ref{lem:le2} implies
  \begin{equation}
\partial_t \langle M \rangle_t \ge \frac12(\log N)^{-C_K} n_t.
%\frac{n_t(\log N)^{-C_K}}{2}.
 \end{equation}
Since $0\leq \delta X_k(t)\leq 2\max( \| \bX_{t}^{\wedge}\|_{\infty},
     \|\bX_{t}^{x}\|_{\infty})$, on the event $\cD$ we get
  \begin{equation}
n_t\geq \frac1{2\sqrt{N}(\log N)^2}\sum_{k=1}^{N-1} 
\delta X_k(t)\ind_{\{\delta X_k(t)\geq a\}} \geq \frac{A_t}{4\sqrt{N}(\log N)^2}.
 \end{equation}
\end{proof}

\begin{proof}[Proof of Lemma \ref{lem:le2}]
We write $\rho^\wedge_k=\rho_{X^\wedge_{k-1},X^\wedge_{k+1}}$, $\rho^x_k=\rho_{X^x_{k-1},X^x_{k+1}}$ for the resampling densities at $k$. Define
$$q_k:=\frac12 \int_\bbR|\rho^\wedge_k(u)-\rho^x_k(u)|du \;.$$
Recall the sticky coupling of Subsection \ref{Subsec:sticky}, in particular the laws $\nu_i$ defined therein. The derivative of the angle bracket
$\partial_t \langle M \rangle_t$ admits an explicit expression which can be derived from the sticky coupling description. For any $t \in [\cT_{i-1}, \cT_{i} \wedge \cR)$
 \begin{equation}\label{drifft}
 \partial_t \langle M \rangle_t = \sum_{k=1}^{N-1}\left(
 (1-q_k) (\delta X_k(t_-))^2 + q_k \bbE[Y^2\,|\,\cF_{t_-}] \right),
 \end{equation}
where 
$$
Y=(Y^\wedge  - Y^x - \delta X_k(t_-))\wedge (R-M_{t-})\,,\quad R:= N^{3/2-(i-2)\eta}
$$
and
$(Y^\wedge,Y^x)$ are, conditionally given $\cF_{t-}$, independent r.v.~with densities $\nu_3$ and $\nu_1$ respectively. The expression \eqref{drifft} simply comes from the fact that for each $k$,
$M_t$ will jump by an amount $\delta X_k(t_-)$ with probability $1-q_k$ and by an amount $Y$ with probability $q_k$. Note that the truncation with $R-M_{t-}$ in the variable $Y$ comes from the definition of $M$ in terms of $A$.

We now work on the event $\cB \cap \cC$. From Lemma \ref{lem:lequ} we have
\begin{equation}\label{eq:qq1}
q_k\leq C (\delta\bar X_k)(\log N)^{2K},
 \end{equation}
 for all $k$, where we use the notation 
 $$
 \delta\bar X_k=\frac12(X^\wedge_{k+1} +X^\wedge_{k-1} -
X^x_{k+1}- X^x_{k-1}).
$$ 
To prove Lemma \ref{lem:le2} it is then sufficient to show that if $q_k\geq 1/2$ then 
\begin{equation}\label{eq:tos}
\bbE\left[Y^2\,|\,\cF_{t_-}\right]\geq (\log N)^{-C_K}\;,
 \end{equation}
for some constant $C_K>0$. Note that under the event $\cC$ we have $ R-M_{t-}\geq R/2$. Moreover, if $q_k \ge 1/2$, because of the event $\cB$ by Lemma \ref{lem:lerho} the density of the random variable $\tilde{Y}:=Y^\wedge  - Y^x - \delta X_k(t_-)$ is bounded above by $L:=C'(\log N)^{2K}$. 
We next observe that we may assume $R\geq 2$. Indeed, if $R\leq 2$ and 
$q_k\geq 1/2$, then by \eqref{eq:qq1} we also have $\delta \bar X_k \geq (2C)^{-1} (\log N)^{-2K}$ and
$$(2C)^{-1} (\log N)^{-2K} \le M_{t-} \le N^{3/2-(i-3/2)\eta} = N^{-\eta/2} R \le 2 N^{-\eta/2}\;,$$
thus raising a contradiction. Hence assuming $ R-M_{t-}\geq R/2$ and $R\geq 2$ we may estimate
\begin{align*}
\bbE\left[Y^2\,|\,\cF_{t_-}\right]&\geq\bbE\left[\tilde{Y}^2\wedge (R/2)^2\,|\,\cF_{t_-}\right]\geq \int_0^{1} 2v\bbP(|\tilde{Y}|>v\,|\,\cF_{t_-})dv.
\end{align*}
The bounded density property implies $\bbP(|\tilde{Y}|>v\,|\,\cF_{t_-})\ge 1-2Lv\geq 1/2$ for all $v\in[0,(4L)^{-1}]$. It follows that
\begin{align*}
\bbE\left[Y^2\,|\,\cF_{t_-}\right]\geq\int_0^{(4L)^{-1}} vdv =\frac1{32 L^{2}}.  
\end{align*}  
This proves \eqref{eq:tos}.
\end{proof}

\section{Proof of Theorem \ref{th:wedge}}\label{Sec:Wedge}
The proof is based on ideas first introduced in \cite{Lac16} for card shuffling by adjacent transpositions. 
An adaptation to the continuous setting was later developed in \cite{CLL}, for the specific case of the adjacent walk on the simplex. Here we are going to follow the proof of \cite[Proposition 14]{CLL}, with some minor modifications due to the different setting. 
We start by recalling the Peres-Winkler censoring inequality.

\subsection{Censoring}\label{sec:censor}
The censoring inequality of Peres and Winkler \cite{PWcensoring} compares the distance to equilibrium at time $t$ for two Markov processes, one of which is obtained as a censored version of the other by omitting some of the updates according to a given censoring scheme. The version of the result that we need here is formulated as Proposition \ref{pro:censor} below. The proof is an 
adaptation to the present 
setting of the original argument for monotone finite spins systems in \cite{PWcensoring}. For completeness we give a brief self-contained account below.

A {\em censoring scheme} $\cC$ is defined as a c\`adl\`ag map
$$
\cC:[0,\infty)\mapsto \cP(\{1,\dots,N-1\}),
$$
where $\cP(A)$ denotes the set of all subsets of a set $A$. The subset $\cC(s)$, at any time $s\geq 0$, represents the set of labels whose update is to be suppressed at that time. More precisely, given a censoring scheme $\cC$, and an initial condition $x\in\gO_N$, we write $P_{t,\cC}^x$ for the law of the random variable obtained by starting at $x$ and applying the standard graphical construction (see Section \ref{Sec:GC}) with the proviso
that if label $j$ rings at time $s$, then the update is performed if and only if $j\notin\cC(s)$.  In particular, the uncensored evolution $P^x_t$ corresponds to $P^x_{t,\cC}$ when $\cC(s)\equiv\eset$. Given a distribution $\mu$ on $\gO_N$, we write $$\mu P_{t,\cC}=\int  P^x_{t,\cC}\,\mu(dx).$$ 
Let $\cS_N$ denote the set of probability measures $\mu$ on $\gO_N$ which are absolutely continuous with respect to \ $\pi_N$ and such that the density $d\mu/d\pi_N$ is an increasing function on $\gO_N$. Recall the notation $\mu\leq \nu$ for stochastic domination.

\begin{proposition}\label{pro:censor}
If $\mu\in\cS_N$, and $\cC$ is a censoring scheme, then for all $t\geq 0$
\begin{equation}\label{Eq:censura}
\|\mu P_t-\pi_N\|_{TV}\leq \|\mu P_{t,\cC}-\pi_N\|_{TV}.
\end{equation}
\end{proposition}
The proof is a consequence of the next two lemmas.

\begin{lemma}\label{lem:kfs}
If $\mu,\nu$ are two probability measures on $\gO_N$ such that $\mu\in\cS_N$ and $\mu\leq \nu$, then 
 \begin{align}\label{Eq:familyS2}
 \|\mu-\pi_N\|_{TV}\leq \|\nu-\pi_N\|_{TV}.
\end{align}
\end{lemma}
\begin{proof}
Setting $\varphi=d\mu/d\pi_N$, and $A=\{\varphi\geq 1\}$, %the total variation distance satisfies 
\begin{align}
\|\mu-\pi_N\|_{TV} &= \mu(A)-\pi_N(A).
\end{align}
Since $A$ is increasing, $\mu(A)\leq \nu(A)$, and therefore
\begin{align}
\|\mu-\pi_N\|_{TV} &\leq \nu(A) - \pi_N(A)\leq \|\nu-\pi_N\|_{TV}\,.
\end{align}
\end{proof} 
Let $\cQ_i:L^2(\gO_N,\pi_N)\mapsto L^2(\gO_N,\pi_N)$, $i=1,\dots,N-1$, denote the 
integral operator
\begin{align}
\label{eq:qif}
\cQ_i f (x) = \int f(x^{(i,u)})\rho_{x_{i-1},x_{i+1}}(u)du,
\end{align}
so that $\cQ_i f $ is the expected value of $f$ after the update of label $i$; see \eqref{projectors}. 
If $\mu$ is a probability on $\gO_N$, 
we write $\mu \cQ_i$ for the probability measure defined by
$$
\mu \cQ_i (f) = \int \mu(dx)\cQ_i f(x)\,.
$$
\begin{lemma}\label{lem:muqi}
 
If $\mu  \in\cS_N$ then $\mu \cQ_i\in\cS_N$ and $\mu \cQ_i\leq \mu$, for all $i=1,\dots,N-1$.

\end{lemma}
\begin{proof}
Set $\varphi=d\mu/d\pi_N$. Then $\mu \cQ_i$ has density $\cQ_i\varphi$ with respect to \ $\pi_N$. 
Since $\varphi$ is increasing, for any $x,y\in\gO_N$ with $x\leq y$, from \eqref{eq:qif} and Lemma \ref{lem:mongc} (or more precisely \eqref{firstorder})it follows that
\begin{align*}
\cQ_i\varphi(x) &\leq \cQ_i\varphi(y)\,.
\end{align*}
Therefore $\mu\cQ_i\in\cS_N$.
To prove the stochastic domination $\mu \cQ_i\leq \mu$, 
we show that $\mu \cQ_i (g)\leq \mu (g)$ for any bounded measurable  increasing function $g$. 
Notice that
$$
\mu \cQ_i (g) = \pi_N\left[\varphi\cQ_ig\right] = \pi_N\left[(\cQ_i\varphi) (\cQ_ig)\right].
$$ 
Since $\varphi$, $g$ are increasing, the FKG inequality on $\bbR$, which is valid for any probability measure, implies that $(\cQ_i\varphi) (\cQ_ig)\leq \cQ_i(\varphi g)$ pointwise. Therefore,
$$
\mu \cQ_i (g) \leq \pi_N\left[\cQ_i(\varphi g)\right]
 = \pi_N\left[\varphi g\right] = \mu (g).
$$ 
\end{proof}

\begin{proof}[Proof of Proposition \ref{pro:censor}]
By Lemma \ref{lem:kfs} it is sufficient to prove that $\mu P_t\in \cS_N$  and $\mu P_t\leq \mu P_{t,\cC}$
for all $t$. 
By conditioning on the realization $\cT_t$ of the Poisson clocks $\cT^{(j)}$, $j\in\lint 1,N-1\rint $ up to time $t$ in the graphical construction,  
the uncensored %and the censored 
evolution at time $t$ has a distribution of the form 
\begin{align}\label{Eq:muz}
\mu^z=\mu\, \cQ_{z_1}\cdots \cQ_{z_n}\,,
\end{align}
where $z:=(z_1,\dots,z_n)\in\lint1,N-1\rint^n$ is a fixed sequence, while the censored 
evolution at time $t$ has distribution of the form $\mu^{z'}$, where 
$z'$ denotes a sequence obtained from $z$ by removing some of its entries. 
Taking the expectation over $\cT_t$ then shows that 
it is sufficient to prove that $\mu^z\in\cS_N$ and $\mu^{z}\leq \mu^{z'}$ for any pair of such sequences $z,z'$.  
Lemma \ref{lem:muqi} shows that $\mu^z\in \cS_N$ for any $\mu\in\cS_N$ and any sequence $z$. To prove 
$\mu^{z}\leq \mu^{z'}$ we may restrict to the case where $z$ and $z'$ differ by the removal of a single update, say $z_j$, so that 
$$
z=(z_1,\dots,z_{j-1},z_j,z_{j+1},\dots,z_n)\,,\quad z'=(z_1,\dots,z_{j-1},z_{j+1},\dots,z_n).
$$
Let $\mu_1 = \mu \cQ_{z_1}\cdots \cQ_{z_{j}}$, and $\mu_2 = \mu \cQ_{z_1}\cdots \cQ_{z_{j-1}}$. Then $\mu_1=\mu_2\cQ_{z_{j}}$ and thus, by Lemma \ref{lem:muqi} one has  $\mu_1\leq \mu_2$.
Moreover,  
$$
\mu^z = \mu_1\cQ_{z_{j+1}}\cdots \cQ_{z_{n}} \leq \mu_2\cQ_{z_{j+1}}\cdots \cQ_{z_{n}} = \mu^{z'}\,, 
$$ 
where the inequality follows from the fact that each update preserves the monotonicity, (cf. Equation \eqref{firstorder}). 
\end{proof}

\subsection{Relaxation of skeletons}\label{sec:skeleton}
For any integer $K\geq 2$, consider the $K-1$ labels $u_i:=\lfloor iN/K\rfloor$, $i=1,\dots,K-1$. We consider the evolution of the heights 
\begin{align}\label{Eq:topk}
Y_i(t) = X_{u_i}(t)\,,\qquad i=1,\dots,K-1,
\end{align}
which will be referred to as the $K$-{\em skeleton} of the interface ${\bf X}(t)$. 
\begin{proposition}\label{prop:skeleton}
Fix an integer $K\geq 2$. Let $\mu_t= P^\wedge_t$ and let $\bar \mu_t$ denote the marginal of $\mu_t$ on the $K$-skeleton $\{Y_i(t), i=1,\dots,K-1\}$. If $\bar\pi_N$ denotes the corresponding equilibrium distribution, then for any fixed $\gd>0$, with $t_\gd = (1+\gd)\frac{\log N}{2\gap_N}$ one has
\begin{equation}\label{Eq:top1}
\lim_{N\to\infty}\|\bar \mu_{t_\gd} - \bar\pi_N\|_{TV} = 0. 
\end{equation}
\end{proposition}
Following \cite{Lac16}, the proof of Proposition \ref{prop:skeleton} is based on a subtle use of the FKG inequality together with an explicit estimate on the expected value of the variables $Y_i(t)$.  
Given a probability $\mu$ on $\gO_N$, we write $\bar \mu$ for the marginal of $\mu$ on the $K$-skeleton $y:=(y_1,\dots,y_{K-1})$, where $y_i=x_{u_i}$ for each $i=1,\dots,K-1$. 

We  use the following notation for the area associated to $K$-skeleton variables $y_i = x_{u_i}$:
$$
 W= \sum_{i=1}^{K-1}y_i,
$$
and write $\mu(W)= \bar \mu(W)$ for the expected value of $W$ under $\mu$. 

\begin{proposition}\label{prop:muW}
For any $\gep>0$, $K\geq 2$, there exists $\eta=\eta(K,\gep)>0$ such that for all $N\geq 2$, $\mu\in\cS_N$ one has:
\begin{equation}\label{Eq:contr1}
\mu(W)\leq\eta \sqrt N\;\;\;\Rightarrow\;\;\;\|\bar \mu -\bar \pi_N\|_{TV}\leq \gep.
\end{equation}
\end{proposition}
The proof of Proposition \ref{prop:muW} is omitted since it is identical to the proof of Proposition 36 in \cite{CLL}. Let us however point out that this proof uses in a crucial way the improved FKG inequality \eqref{piapib} in Proposition \ref{prop:fkg}.

Next, we control the expected value of $W$ at time $t$. Let ${\bf X}^\wedge(t)=\{X^\wedge_k(t)\}$ 
denote the random variables with joint law $ P^\wedge_t$.
 \begin{proposition}\label{prop:Wt}
For any $k=1,\dots,N-1$, any $t\geq 0$:
$$
\bbE\left[X^\wedge_k(t)\right]\leq 12Ne^{-\gap_Nt}.
$$
In particular, if $\mu_t= P^\wedge_t$, then for all $t\geq 0$: 
\begin{equation}\label{Eq:contrW1}
\mu_t(W)\leq 12KN e^{-\gap_Nt}.
\end{equation}
\end{proposition}
\begin{proof}
Set $v(t)=(v_1(t),\dots,v_{N-1}(t))$, where $v_k(t)=\bbE\left[X^\wedge_k(t)\right]$.
Expanding $v_k(t)$ in the orthonormal basis \eqref{Eq:eigenfcts},
one finds 
$v_k(t) = \sum_{j=1}^{N-1} a_j(t)  \varphi_{k}^{(j)}$, where 
$a_j(t)%=(\varphi_j,v(t))
=\sum_{k=1}^{N-1}\varphi_{k}^{(j)}v_k(t)$. Since $\frac{d}{dt}v_k(t)=\frac12(\Delta v(t))_k$, it follows that 
$$
a_j(t) = a_j(0)e^{-\gl_j t}\,,\qquad a_j(0) =\sum_{k=1}^{N-1}
\varphi_{k}^{(j)}v_k(0).
$$
In particular, $|a_j(0)|\leq  \sqrt{ 2N}|v(0)|_\infty$, where $|v(0)|_\infty=\max_k v_k(0)$. Therefore,
\begin{equation}\label{Eq:diagD1}
v_k(t)\leq 2|v(0)|_\infty\sum_{j=1}^{N-1}e^{-\gl_j t}\,.
\end{equation}
Let us show that $|v(0)|_\infty\leq 3N$ for all $N$ large enough. Raising the boundary condition from $(0,0)$ to $(2N,2N)$ and using   monotonicity, we see that for all $k$ 
the random variable $X_k$ with distribution $\nu^\wedge$ is stochastically dominated by the random variable $X_k+2N$ where $X_k$ has distribution $\pi(\cdot\,|\,\min_i x_i\geq -N)$. The claimed monotonicity with respect to the boundary conditions can be checked using the FKG inequality for $\pi(\cdot\,|\,\min_i x_i\geq n)$.
Indeed the density of the measure with raised boundary with respect to the original one is equal (up to a renormalizing constant) to 
$$e^{V(2N+x_1)-V(x_1)+V(2N+x_{N-1})-V(x_{N-1})}$$
which by convexity of $V$ is increasing for the order ``$\le$'' on $\gO_N$.
It follows that  
\begin{equation}\label{Eq:monobc}
v_k(0) \leq 2N + \pi(x_k|\,{\min}_ix_i\geq -N)\,.
\end{equation}
From Corollary \ref{cor:bazics} and the union bound, 
\begin{equation}\label{Eq:expot1}
 \pi({\min}_i x_i\geq -N)\geq 1-Ne^{-cN},
\end{equation}
for some constant $c>0$ and all $N$ large enough. Moreover, Lemma \ref{compare} also shows that, uniformly in $k$,
\begin{equation}\label{Eq:expot11}
 \pi(x_k; {\min}_i x_i\geq -N)\leq \pi(x_k^2)^\frac12\leq C\sqrt N\,,
\end{equation}
for some constant $C>0$ and all $N$ large enough. The estimates \eqref{Eq:monobc}-\eqref{Eq:expot11} imply $|v(0)|_\infty\leq 3N$ for $N$ large.
From \eqref{Eq:diagD1},
using $\gl_j \geq j\gl_1$ it follows that
$$
v_k(t)\leq \frac{6Ne^{-\gl_1 t}}{1-e^{-\gl_1 t}}.
$$ 
If $t$ is such that $e^{-\gl_1 t}\leq 1/4$ then this implies $v_k(t)\leq 8Ne^{-\gl_1 t}$. On the other hand if 
$e^{-\gl_1 t}\geq 1/4$ then, using the monotonicity $P^\wedge_t\leq \nu^\wedge$ one has $$v_k(t)\leq v_k(0)\leq 3N \leq 12Ne^{-\gl_1 t}.$$ Since $\gl_1=\gap_N$, this proves the desired upper bound. 
\end{proof}

\begin{proof}[Proof of Proposition \ref{prop:skeleton}]
Proposition \ref{prop:Wt} shows that
\begin{equation}\label{Eq:top2}
\lim_{N\to\infty}\frac{\bar \mu_{t_\gd}(W)}{\sqrt N} = 0,
\end{equation}
and Proposition \ref{prop:muW} shows that \eqref{Eq:top2} is sufficient to achieve the desired convergence of $K$-skeletons.  
\end{proof}
 \subsection{Relaxation of the censored dynamics}\label{sec:censoreddyn}
Consider the censored process obtained by suppressing all updates of the skeleton variables. That is, we use the censoring scheme $\cC$ such that $\cC(s)=\{u_1,\dots,u_{K-1}\}$, $s\geq 0$.
\begin{proposition}\label{prop:special2}
Let $P_{t,\cC}^x=\gd_x P_{t,\cC}$ and let $\pi_N(\cdot|y)$ denote the equilibrium distribution given the skeleton heights  $y_i=x_{u_i}, i=1,\dots,K-1$. For any $\gd\in(0,1)$,  define $K=\lfloor \gd^{-1}\rfloor$ and $s_\gd = \gd\frac{\log N}{2\gap_N}$, and let $B_{N,\delta}$ denote the event
  \begin{equation}\label{Eq:top0bn}
B_{N,\delta}=\Big\{x\in \Omega_N:\, \|x\|_\infty \le 2N,\; \max_{i=1,\dots,K} |x_{u_i}| \leq N/2K\Big\}.
\end{equation}
Then there exists $\delta_0\in(0,1)$ such that for all fixed $\delta\in(0, \delta_0)$ and for all $N$ sufficiently large:
\begin{equation}\label{Eq:top10}
\sup_{x\in B_{N,\delta}}\| P_{s_\gd,\cC}^x- \pi_N(\cdot|y)\|_{TV} \leq  \delta\,.
\end{equation}
\end{proposition}
\begin{proof}

The censored process is a collection of $K$ independent processes each describing the evolution of an interface on a segment of length $n:=\lfloor N/K\rfloor$, with fixed boundary heights $(y_{i-1},y_i)$, where $y_i=x_{u_i}$. If  $ x\in B_{N,\delta}$ then the left and right boundary conditions of each interface satisfy $$|y_{i-1}-y_i|\leq N/K\leq 2n.$$
Moreover, if  $x\in B_{N,\delta}$ then  the initial condition satisfies $\|x\|_\infty\leq 2N\leq n^2$, if $N$ is large enough.   
%on the mixing time 
From the mixing time bound given in Corollary \ref{th:corol} (see Remark \ref{rem:uniformity}) it follows that  for any given $\gep\in (0,1)$, when $N$ is sufficiently large, each individual process has $\gep$-mixing time bounded above by  
\begin{equation}\label{Eq:crudemix}
C\, n^2\log n\le \frac{C }{K^2}\,N^2\log (N) \le  s_{\gd}\,,
\end{equation}
if $\delta > 0$ is small enough. Thus the entire censored process  satisfies  
 $$
\| P_{s_\gd,\cC}^x- \pi_N(\cdot|y)\|_{TV} \leq K \gep\,. %\le \delta\;,
$$
The claimed inequality follows by taking $\gep=K^{-1}\gd$.
\end{proof}

\subsection{Proof of Theorem \ref{th:wedge}}

We want to prove that for any $\gd>0$, 
$$ \lim_{N\to\infty}\| P_{t_\gd}^\wedge - \pi_N \|_{TV} =0\;,$$
where $ t_\delta=(1+\gd)\frac{\log N}{2\gap_N}$. 
Set $K=\lfloor \gd^{-1}\rfloor$ and let $\cC'$ denote the censoring scheme defined by
$\cC'(s)=\eset$ for $s\in[0, t_{\gd/2})$ and $\cC'(s)=\{u_1,\dots,u_{K-1}\}$ for $s\geq t_{\gd/2}$.
Let also $P^\wedge_{t,*}=P^\wedge_{t,\cC'}$ denote the corresponding censored process. From Proposition \ref{pro:censor} we have
$$
\| P_{t_\gd}^\wedge - \pi_N \|_{TV} \leq \| P_{t_\gd,*}^\wedge - \pi_N \|_{TV}.
$$
We are going to construct a coupling of $P_{t_\gd,*}^\wedge$ and $\pi_N$. 
We first couple the skeleton heights at time $t_{\gd/2}$. 
Set $\mu=P^\wedge_{t_{\gd/2}}$,  
and let $\bbP$ denote a coupling of $\mu$ and $\pi_N$. Let $(X,Z)$ denote the corresponding height variables, so that $X$ has distribution $\mu$ and $Z$ has distribution $\pi_N$. 
The coupling $\bbP$ can be chosen in such a way that the skeleton variables are optimally coupled, that is
$$
\bbP(X_{u_i}=Z_{u_i}, \;i=1,\dots,K-1) =1- \|\bar \mu_{t_{\gd/2}} - \bar\pi_N\|_{TV}.
$$ 
Consider the event $$E=\{x\in\Omega_N: \,|x|_\infty\leq  2N\}.$$
Monotonicity implies that $\pi_N\leq \mu\leq \nu^\wedge$ and therefore 
\begin{align}\label{Eq:xnotine}
\mu(E^{\cc})&\leq N \max_i \mu(|x_i|>2N)\\&
\leq N \max_i\pi_N(x_i<-2N) + N \max_i\nu^\wedge(x_i>2N).
\end{align}
Corollary \ref{cor:bazics} implies
\begin{equation}\label{Eq:expott}
 \max_{i\in \lint 1,N-1\rint}\pi_N(x_i<-N)\leq Ce^{-N/C},
\end{equation}
for some constant $C>0$. Raising the boundary condition from $(0,0)$ to $(\frac32N,\frac32N)$ and using   monotonicity, we see that for all $i$ 
the random variable $X_k$ with distribution $\nu^\wedge$ is stochastically dominated by the random variable $X_k+3N/2$ where $X_k$ has distribution $\pi(\cdot\,|\,\min_i x_i\geq -N/2)$. Thus, reasoning as in \eqref{Eq:monobc} one finds
\begin{equation}\label{Eq:monobc1}
\max_{i\in \lint 1,N-1\rint}\nu^\wedge(x_i>2N)\leq  Ce^{-N/C},
\end{equation}
for some constant $C>0$.
Define the event 
$$
\cA=\{X_{u_i}=Z_{u_i}, \,i=1,\dots,K-1\}\cap\{X\in B_{N,\delta}\},$$
where $B_{N,\delta}$ is given in Proposition \ref{prop:skeleton}. Let $F=\{x\in\Omega_N:\, |x_{u_i}|\leq N/2K\}$ so that 
$\{X\in B_{N,\delta}\}=\{X\in E\cap F\}$. 
Then, 
\begin{equation}\label{Eq:evea}
\cA=\{X\in E\}\cap\{Z\in F\}\cap\{X_{u_i}=Z_{u_i}, \,i=1,\dots,K-1\}.
\end{equation}
Therefore,
$$
\bbP(\cA^c)\leq  \|\bar \mu_{t_{\gd/2}} - \bar\pi_N\|_{TV} + \mu(X\notin E) + \pi_N(Z\notin F).
$$
From \eqref{Eq:xnotine}-\eqref{Eq:monobc1} we have $\mu(X\notin E)\leq 2CNe^{-N/C}$.
From Corollary \ref{cor:bazics} and the union bound %and the exponential integrability,
one has that $$\pi_N(Z\notin F)\leq C_1e^{-N/C_1},$$ for some $C_1=C_1(K)>0$ independent of $N$.  

If the event $\cA$ occurs, then 
we couple the interfaces at time $t_\gd =t_{\gd/2}+s_{\gd/2}$ with the optimal coupling attaining the total variation distance $\| P_{s_{\gd/2},\cC}^x- \pi_N(\cdot|y)\|_{TV}$, where $\cC$ is as in Proposition \ref{prop:special2}. This shows that
$$
\| P_{t_\gd,*}^\wedge - \pi_N \|_{TV} \leq \bbP(\cA^c)
+ \sup_{x\in B_{N,\delta}}
\| P_{s_{\gd/2},\cC}^x- \pi_N(\cdot|y)\|_{TV}.
$$
From \eqref{Eq:evea}, Proposition \ref{prop:skeleton} and Proposition \ref{prop:special2}, 
$$
\limsup_{N\to\infty}\| P_{t_\gd}^\wedge - \pi_N \|_{TV} \leq 2\gd.
$$
The distance $\| P_{t_\gd}^\wedge - \pi_N \|_{TV}$ is decreasing as a function of $\gd$, and therefore we may take $\gd\to 0$ in the right hand side above to conclude.

\appendix

\subsection*{Acknowledgements}
P.C.\ thanks University Paris-Dauphine for a funding of ``Professeur Invit\'e'' and IMPA for the hospitality in the early stage of this work. C.L.\ acknowledges support from the grant SINGULAR ANR-16-CE40-0020-01.
This work was realized in part during H.L.\ extended stay in Aix-Marseille University funded by the European Union’s Horizon 2020 research and innovation programme under the Marie Skłodowska-Curie grant agreement No 837793.
\bibliographystyle{Martin}
\bibliography{library}

\begin{thebibliography}{CMT12}
\expandafter\ifx\csname url\endcsname\relax
  \def\url#1{\texttt{#1}}\fi
\expandafter\ifx\csname urlprefix\endcsname\relax\def\urlprefix{URL }\fi
\expandafter\ifx\csname href\endcsname\relax
  \def\href#1#2{#2}\fi
\expandafter\ifx\csname burlalt\endcsname\relax
  \def\burlalt#1#2{\href{#2}{\texttt{#1}}}\fi

\bibitem[BIV00]{BodineauIoffeVelenik}
\textsc{T.~Bodineau}, \textsc{D.~Ioffe}, and \textsc{Y.~Velenik}.
\newblock Rigorous probabilistic analysis of equilibrium crystal shapes.
\newblock \emph{Journal of Mathematical Physics} \textbf{41}, no.~3, (2000),
  1033--1098.

\bibitem[BM13]{BartheMilman}
\textsc{F.~Barthe} and \textsc{E.~Milman}.
\newblock Transference principles for log-sobolev and spectral-gap with
  applications to conservative spin systems.
\newblock \emph{Communications in Mathematical Physics} \textbf{323}, no.~2,
  (2013), 575--625.

\bibitem[BW09]{BartheWolff}
\textsc{F.~Barthe} and \textsc{P.~Wolff}.
\newblock Remarks on non-interacting conservative spin systems: the case of
  gamma distributions.
\newblock \emph{Stochastic processes and their applications} \textbf{119},
  no.~8, (2009), 2711--2723.

\bibitem[Cap03]{Cap}
\textsc{P.~Caputo}.
\newblock Uniform poincar{\'e} inequalities for unbounded conservative spin
  systems: the non-interacting case.
\newblock \emph{Stochastic processes and their applications} \textbf{106},
  no.~2, (2003), 223--244.

\bibitem[CLL20]{CLL}
\textsc{P.~{Caputo}}, \textsc{C.~{Labb{\'e}}}, and \textsc{H.~{Lacoin}}.
\newblock {Mixing time of the adjacent walk on the simplex}.
\newblock \emph{Ann. Probab.}  to appear.

\bibitem[CMT12]{CapMarTon}
\textsc{P.~Caputo}, \textsc{F.~Martinelli}, and \textsc{F.~L. Toninelli}.
\newblock Mixing times of monotone surfaces and sos interfaces: a mean
  curvature approach.
\newblock \emph{Communications in Mathematical Physics} \textbf{311}, no.~1,
  (2012), 157--189.

\bibitem[Dur19]{Dur19}
\textsc{R.~Durrett}.
\newblock \emph{Probability: Theory and Examples}.
\newblock Cambridge Series in Statistical and Probabilistic Mathematics.
  Cambridge University Press, 5 ed., 2019.

\bibitem[DZ09]{DZ09}
\textsc{A.~Dembo} and \textsc{O.~Zeitouni}.
\newblock \emph{Large Deviations Techniques and Applications}.
\newblock Stochastic Modelling and Applied Probability. Springer Berlin
  Heidelberg, 2009.

\bibitem[Fun05]{Funaki}
\textsc{T.~Funaki}.
\newblock Stochastic interface models.
\newblock \emph{Lectures on Probability Theory and Statistics, Ecole d'Ete de
  Probabilites de Saint-Flour XXXIII-2003} (2005).

\bibitem[Gia02]{Giacomin}
\textsc{G.~Giacomin}.
\newblock \emph{{Aspects of statistical mechanics of random surfaces}}.
\newblock Lecture Notes for course given at IHP. 2002.

\bibitem[Lac16]{Lac16}
\textsc{H.~Lacoin}.
\newblock Mixing time and cutoff for the adjacent transposition shuffle and the
  simple exclusion.
\newblock \emph{Ann. Probab.} \textbf{44}, no.~2, (2016), 1426--1487.
\newblock
  \burlalt{doi:10.1214/15-AOP1004}{http://dx.doi.org/10.1214/15-AOP1004}.

\bibitem[Lig05]{Liggettbook}
\textsc{T.~M. Liggett}.
\newblock \emph{Interacting particle systems}.
\newblock Classics in Mathematics. Springer-Verlag, Berlin, 2005.
\newblock Reprint of the 1985 original.

\bibitem[LL19]{labbe2019cutoff}
\textsc{C.~Labb{\'e}} and \textsc{H.~Lacoin}.
\newblock Cutoff phenomenon for the asymmetric simple exclusion process and the
  biased card shuffling.
\newblock \emph{Annals of Probability} \textbf{47}, no.~3, (2019), 1541--1586.

\bibitem[LPW17]{LevPerWil}
\textsc{D.~A. Levin}, \textsc{Y.~Peres}, and \textsc{E.~L. Wilmer}.
\newblock \emph{Markov chains and mixing times}.
\newblock American Mathematical Society, Providence, RI, 2017.
\newblock Second edition of [ MR2466937], With a chapter on ``Coupling from the
  past'' by James G. Propp and David B. Wilson.

\bibitem[MO13]{OttoMenz}
\textsc{G.~Menz} and \textsc{F.~Otto}.
\newblock Uniform logarithmic sobolev inequalities for conservative spin
  systems with super-quadratic single-site potential.
\newblock \emph{The Annals of Probability} \textbf{41}, no.~3B, (2013),
  2182--2224.

\bibitem[MS12]{MartinelliSinclair}
\textsc{F.~Martinelli} and \textsc{A.~Sinclair}.
\newblock Mixing time for the solid-on-solid model.
\newblock \emph{The Annals of Applied Probability} \textbf{22}, no.~3, (2012),
  1136--1166.

\bibitem[Pet75]{Petrov}
\textsc{V.~V. Petrov}.
\newblock \emph{Sums of independent random variables}.
\newblock Springer-Verlag, New York-Heidelberg, 1975.
\newblock Translated from the Russian by A. A. Brown, Ergebnisse der Mathematik
  und ihrer Grenzgebiete, Band 82.

\bibitem[Pos97]{Posta}
\textsc{G.~Posta}.
\newblock Spectral gap for an unrestricted kawasaki type dynamics.
\newblock \emph{ESAIM: Probability and Statistics} \textbf{1}, (1997),
  145--181.

\bibitem[Pre74]{Preston}
\textsc{C.~J. Preston}.
\newblock A generalization of the {${\rm FKG}$} inequalities.
\newblock \emph{Comm. Math. Phys.} \textbf{36}, (1974), 233--241.

\bibitem[PW13]{PWcensoring}
\textsc{Y.~Peres} and \textsc{P.~Winkler}.
\newblock Can extra updates delay mixing?
\newblock \emph{Communications in Mathematical Physics} \textbf{323}, no.~3,
  (2013), 1007--1016.

\bibitem[Vil39]{ville1939}
\textsc{J.~Ville}.
\newblock {\'E}tude critique de la notion de collectif (1939).

\bibitem[Wil04]{Wil04}
\textsc{D.~B. Wilson}.
\newblock Mixing times of {L}ozenge tiling and card shuffling {M}arkov chains.
\newblock \emph{Ann. Appl. Probab.} \textbf{14}, no.~1, (2004), 274--325.
\newblock
  \burlalt{doi:10.1214/aoap/1075828054}{http://dx.doi.org/10.1214/aoap/1075828054}.

\end{thebibliography}

\end{document}